\documentclass[preprint,12pt]{elsarticle}
\usepackage{psfrag}
\usepackage{a4wide}
\usepackage{rotating}
\usepackage{color}
\usepackage{amssymb}
\usepackage{amsmath}
\usepackage{amsthm}
\usepackage[lined,boxed,linesnumbered,commentsnumbered]{algorithm2e}
\usepackage{verbatim}

\newcommand{\ab}[1]{\boldsymbol{#1}}
\def\bfm#1{\boldsymbol{#1}}
\newcommand{\f}[1]{\mathbf{#1}}
\newcommand{\bb}[1]{\bfm{#1}}
\newcommand{\N}{\mathbb N}

\newcommand{\R}{\mathbb R}

\newcommand{\s}{\scriptstyle}

\newtheorem{thm}{Theorem}
\newtheorem{lem}[thm]{Lemma}

\newtheorem{obs}[thm]{Observation}

\theoremstyle{definition}
\newtheorem{ex}[thm]{Example}
\newtheorem{defn}[thm]{Definition}
\newtheorem{rem}[thm]{Remark}
\newproof{pf}{proof}

\advance\textheight by 0.4cm
\advance\topmargin by -0.2cm

\begin{document}

\begin{frontmatter}

\title{Space of $C^2$-smooth geometrically continuous isogeometric functions on planar multi-patch geometries: Dimension and numerical experiments}


\author[pavia,lnz]{Mario Kapl\corref{cor}}
\ead{mario.kapl@ricam.oeaw.ac.at}
 
\author[slo1,slo2]{Vito Vitrih}
\ead{vito.vitrih@upr.si}

\address[pavia]{Dipartimento di Matematica ``F.Casorati", Universit\`{a} degli Studi di Pavia, Italy}

\address[lnz]{Johann Radon Institute for Computational and Applied Mathematics, \\Austrian Academy of Sciences, Linz, Austria}

\address[slo1]{IAM and FAMNIT, University of Primorska, Koper, Slovenia}

\address[slo2]{Institute of Mathematics, Physics and Mechanics, Ljubljana, Slovenia}

\cortext[cor]{Corresponding author}

\begin{abstract}
We study the space of $C^{2}$-smooth isogeometric functions on bilinearly parameterized multi-patch domains~$\Omega \subset \R^{2}$, where the graph of each 
isogeometric function is a multi-patch spline surface of bidegree~$(d,d)$, $d \in \{5,6 \}$. The space is fully characterized by the equivalence of the 
$C^2$-smoothness of an isogeometric function and the $G^2$-smoothness 
of its graph surface (cf. \cite{Pe15, KaViJu15}). 
This is the reason to call its functions $C^{2}$-smooth geometrically continuous isogeometric functions.   

In particular, the dimension of this $C^{2}$-smooth isogeometric space is investigated. The study is based on the decomposition of the space into 
three subspaces and is an extension of the work~\cite{KaplVitrih2016} to the multi-patch case. In addition, we present an algorithm for the 
construction of a basis, and use the resulting globally $C^{2}$-smooth functions for numerical experiments, such as performing $L^{2}$ approximation and solving 
triharmonic equation, on bilinear multi-patch domains. The numerical results indicate optimal approximation order.
\end{abstract}


\end{frontmatter}

%%%%%%%%%%%%%%%%%%%%%%%%%%%%%%%%%%%%%%%%%%%%%%%%%%%%%%%%%%%
\section{Introduction} \label{sec:introduction}
%%%%%%%%%%%%%%%%%%%%%%%%%%%%%%%%%%%%%%%%%%%%%%%%%%%%%%%%%%%

The concept of isogeometric analysis \cite{ANU:9260759, CottrellBook, HuCoBa04} can be employed to solve high order PDEs (i.e., partial differential equations), see e.g. 
\cite{Dede2012, GoBaHu08, KiBaHsWuBl10, KiBlLiWu09, LiDeEvBoHu13} for $4$-th order PDEs and \cite{BaDe15, Gomez2012, TaDe14} for $6$-th order PDEs. 
The main idea is to use the same spline function space to represent the physical domain and to describe the solution space. Thereby, the high order PDEs 
are solved via their weak forms using standard Galerkin discretization, see e.g.~\cite{BaDe15, TaDe14}. 
This requires in general spline function spaces of $C^{1}$-smoothness or even higher. 

The investigation of such smooth isogeometric spline spaces and especially their basis construction are non-trivial tasks for multi-patch domains with possibly 
extraordinary vertices, i.e., vertices with a valency~$\nu \neq 4$. The concept of geometric continuity (cf. \cite{Pe02}) provides a framework to characterize these 
spaces. For a non-negative integer $s$, an isogeometric function is $C^{s}$-smooth on the given multi-patch domain if and only if its graph surface over the multi-patch 
domain is $G^{s}$-smooth (i.e., geometric continuous of order $s$), see \cite{Pe15, KaViJu15}. Therefore, $C^s$-smooth isogeometric functions have been 
called $C^s$-smooth geometrically continuous isogeometric functions~\cite{KaBuBeJu16, KaplVitrih2016, KaViJu15}. 

The study of $C^{1}$-smooth isogeometric spline spaces over multi-patch domains has been started in the last three years. Mainly, 
two approaches are considered depending on the used parameterization of the physical domain. While in the first approach (e.g. \cite{Pe15-2,Peters2,NgKaPe15,NgPe16}) 
the domain parameterizations are $C^{1}$-smooth along the patch interfaces except in the vicinity of extraordinary vertices, in the second approach 
(e.g. \cite{BeMa14,CST15,KaBuBeJu16,KaViJu15, mourrain2015geometrically}) the domain parameterizations belong to a special class of parameterizations, called analysis 
suitable $G^{1}$ multi-patch parameterizations in~\cite{CST15}. In contrast to the first approach, these parameterizations require only $C^{0}$-smoothness along the patch 
interfaces. Amongst others, the class of analysis suitable $G^{1}$ multi-patch parameterizations contains the class of bilinearly parameterized multi-patch domains, 
which has been considered first (cf.~\cite{BeMa14,KaViJu15}). For both concepts numerical results indicate that the resulting $C^{1}$-smooth isogeometric spaces possess 
optimal approximation properties (cf. \cite{CST15,KaBuBeJu16,KaViJu15,NgKaPe15}). In the case of the second approach this is confirmed by theoretical 
investigations~\cite{CST15}.

In the present paper we are interested in the case of $C^{2}$-smooth geometrically continuous isogeometric functions on multi-patch domains. For this we will follow a similar approach as the 
second one above. More precisely, the present paper extends the work~\cite{KaplVitrih2016} for bilinearly parameterized two-patch domains to the case of bilinearly 
parameterized multi-patch domains. We investigate the dimension of the spaces of \emph{$C^2$}-smooth biquintic and bisixtic isogeometric functions on bilinearly 
parameterized multi-patch domains by analyzing all possible configurations of such multi-patch domains. This is done by decomposing the $C^{2}$-smooth isogeometric 
spline space into the direct sum of three subspaces, which are called patch space, edge space and vertex space. Whereas the computation of the dimension of the first two subspaces can 
be seen as a generalization of results in~\cite{KaplVitrih2016} for the two-patch case, the computation of the dimension of the vertex space for all possible 
configurations of bilinear multi-patch domains is a non-trivial task. This can be expected from the results for triangulated domains, cf.~\cite{LaSch07}. 
Our obtained results cover all possible configurations of bilinear multi-patch domains, which is in contrast to the results for the $C^{1}$ case in~\cite{KaBuBeJu16}, 
where only the so-called generic configuration was considered. 
In addition, an algorithm for the construction of a basis for the space of $C^{2}$-smooth geometrically 
continuous isogeometric functions is presented, which is based on the concept of minimal determining sets (cf.~\cite{BeMa14,LaSch07}) for the involved spline 
coefficients. Numerical results indicate that the resulting basis functions are well conditioned.   

A different method similar to the first approach above is described in~\cite{ToSpHiHu16}, where a polar spline configuration is used. 
The idea is to consider a special basis construction in the neighborhood of the extraordinary vertex to achieve the desired smoothness also there.

$C^{2}$-smooth geometrically continuous isogeometric functions are required for solving $6$-th order PDEs over multi-patch domains by means of isogeometric analyis. Two 
relevant examples of $6$-th order PDEs are the triharmonic equation~\cite{BaDe15, KaplVitrih2016, TaDe14} and the Phase-field crystal equation~\cite{Gomez2012}. 
In this paper, we will present examples of solving the triharmonic equation on different bilinear multi-patch domains using our constructed $C^{2}$-smooth geometrically 
continuous isogeometric functions. Furthermore, the numerical results, which are obtained by performing $L^{2}$ approximation on different multi-patch domains, indicate 
analogous to the $C^{1}$ case~\cite{KaBuBeJu16,KaViJu15} and to the $C^{2}$ two-patch case~\cite{KaplVitrih2016} optimal approximation order of the 
considered isogeometric spline spaces. 

The paper is organized as follows. Section~\ref{sec:Geom_IgA_Bilinear} describes the used class of bilinearly parameterized multi-patch domains $\Omega \subset \R^2$, 
and introduce the space of biquintic or bisixtic $C^2$-smooth geometrically continuous isogeometric functions on a given domain $\Omega$. In 
Section~\ref{sec:decompsition} and \ref{sec:dimension}, we study the dimension of this space. For this purpose, we decompose the space of $C^2$-smooth geometrically 
continuous isogeometric functions into the direct sum of three subspaces, which are called patch space, edge space and vertex space, and compute for 
each of these subspaces the dimension by analyzing all possible configurations. Section~\ref{sec:Examples} describes an algorithm for the construction of a basis and presents numerical experiments. 
More precisely, $C^2$-smooth biquintic and bisixtic geometrically isogeometric functions are used for performing $L^2$-approximation to show 
experimentally that the estimated convergence rates are optimal. Furthermore, these functions are used to solve the triharmonic equation over 
different bilinear multi-patch domains. Finally, we conclude the paper.

%%%%%%%%%%%%%%%%%%%%%%%%%%%%%%%%%%%%%%%%%%%%%%%%%%%%%%%%%%%%%%%%%%%%%%%%%%%%%%%%%%%%%%%%%%%%
\section{$C^2$-smooth geometrically continuous isogeometric functions on bilinear multi-patch geometries} \label{sec:Geom_IgA_Bilinear}
%%%%%%%%%%%%%%%%%%%%%%%%%%%%%%%%%%%%%%%%%%%%%%%%%%%%%%%%%%%%%%%%%%%%

%%%%%%%%%%%%%%%%%%%%%%%%%%%%%%%%%%%%%%%%%%%%%%%%%%%%%%%%%%%%%%%%%%%%%%%%%%%%%%%%%%%%%%%%%%%%
\subsection{Bilinearly parameterized multi-patch domains} \label{sec:settings}
%%%%%%%%%%%%%%%%%%%%%%%%%%%%%%%%%%%%%%%%%%%%%%%%%%%%%%%%%%%%%%%%%%%%%%%%%%%%%%%%%%%%%%%%%%%%

Suppose that our computational domain $\Omega \subset \R^2$ consists of $P \in \N, P\geq 2$, mutually disjoint strictly convex quadrangular patches~$\Omega^{(\ell)}$, $\ell=1,2,\ldots, P$, $E$ non-boundary common edges $\Gamma^{(j)}$, $j=1,2,\ldots,E$, and $V$ (inner and boundary) vertices 
\begin{equation}  \label{eq:shape_points}
\ab{v}^{(r)} :=(p_r,q_r)^T
\end{equation}
of valency $\nu_{r} \geq 3$, $r=1,2,\ldots,V$. Note that a boundary vertex of valency two, i.e., a vertex contained in only one 
patch~$\Omega^{(\ell)}$, is never meant by a vertex $\ab{v}^{(r)}$ in this paper. We assume that any two patches have an empty intersection, exactly one common vertex or share 
the whole 
common edge. Moreover, the deletion of any vertex does not split $\Omega$ into subdomains, whose union would be unconnected. We additionally assume that all patches
$\Omega^{(\ell)}$ are images of bijective and regular \emph{bilinear} geometry
mappings~$\ab{G}^{(\ell)}$ 
\begin{align*}
 \ab{G}^{(\ell)}: [0,1]^{2}  \rightarrow \R^{2}, \quad 
 \bb{\xi}^{(\ell)} =(\xi^{(\ell)}_1,\xi^{(\ell)}_2) \mapsto
  (G^{(\ell)}_1,G^{(\ell)}_2)=
 \ab{G}^{(\ell)}(\bb{\xi}^{(\ell)}), \quad \ell =1, 2, \ldots, P,
\end{align*}
i.e., $\Omega^{(\ell)} = \ab{G}^{(\ell)}([0,1]^{2})$.

Let  $\mathcal{S}^{d}_{k}$ denote the tensor-product spline space, spanned by B-splines of bidegree $(d,d)$ 
and knot vectors
\begin{equation*}
 (\underbrace{0,\ldots,0}_{d+1-\mbox{\scriptsize times}},
\underbrace{\textstyle \frac{1}{k+1},\ldots ,\frac{1}{k+1}}_{d-2 - \mbox{\scriptsize times}}, 
\underbrace{\textstyle \frac{2}{k+1},\ldots ,\frac{2}{k+1}}_{d-2 - \mbox{\scriptsize times}},\ldots, 
\underbrace{\textstyle \frac{k}{k+1},\ldots ,\frac{k}{k+1}}_{d-2 - \mbox{\scriptsize times}},
\underbrace{1,\ldots,1}_{d+1-\mbox{\scriptsize times}}),
\end{equation*}
in both parameter directions, where $k \in \N_{0}$ denotes the number of inner knots. 
Clearly $\ab{G}^{(\ell)} \in \mathcal{S}^{d}_{k} \times \mathcal{S}^{d}_{k}$, $\ell=1,2,\ldots, P$, which gives
\begin{equation}  \label{eq:geometry_mapping}
 \ab{G}^{(\ell)}(\bb{\xi}^{(\ell)})=\sum_{\ab{i} \in I} \ab{d}^{(\ell)}_{\ab{i}} N_{\ab{i}}(\bb{\xi}^{(\ell)}), \quad 
 I := \{\bfm{i} = (i_1,i_2); \; \bfm{0} \leq \bfm{i} \leq (d+k(d-2),d+k(d-2))\},
\end{equation}
where $ \ab{d}^{(\ell)}_{\ab{i}} \in \R^{2}$ are the
control points and $N_{\ab{i}}$ are the tensor-product B-splines.  

Throughout the paper, we will only consider degrees $d \in \{5, 6\}$ and assume that the number of inner knots satisfies $k \geq 7 - d$.

%%%%%%%%%%%%%%%%%%%%%%%%%%%%%%%%%%%%%%%%%%%%%%%%%%%%%%%%%%%%%%%%%%%%%%%%%%%%%%%%%%%%%%%%%%%%
\subsection{$C^2$-smooth geometrically continuous isogeometric functions} \label{sec:Geom_IgA}
%%%%%%%%%%%%%%%%%%%%%%%%%%%%%%%%%%%%%%%%%%%%%%%%%%%%%%%%%%%%%%%%%%%%

Recall that $ \ab{G}^{(\ell)}, \, \ell=1,2, \ldots,P$, are contained in $\mathcal{S}^{d}_k \times \mathcal{S}^{d}_k$.
On all patches $\Omega^{(\ell)}$ the
space of isogeometric functions is then given by
\begin{equation*}
 \mathcal{S}^{d}_k \circ (\ab{G}^{(\ell)})^{-1}.
\end{equation*} 
Let us consider the space
\begin{equation*}
V^{(k)} = \left\{ v \in C^{2}(\Omega) : \; v|_{\Omega^{(\ell)}} \in \mathcal{S}^{d}_{k} \circ (\ab{G}^{(\ell)})^{-1}, \; \ell \in \{1,2, \ldots, P\}   \right\},
\end{equation*}
which contains the {\em globally $C^{2}$-smooth isogeometric functions} on $\Omega$, and $k$ denotes the number of inner knots. 
Suppose that $w \in V^{(k)}$.
Then
\begin{equation*} \label{eq:single_function}
(w|_{\Omega^{(\ell)}})(\ab{x}) =  w^{(\ell)}(\ab{x}) = \left(W^{(\ell)} \circ (\ab{G}^{(\ell)})^{-1}\right)(\ab{x}), \quad \ab{x} \in \Omega^{(\ell)}, \;\; \ell=1,2, \ldots, P,
\end{equation*}
with
$$
 W^{(\ell)}(\bb{\xi}^{(\ell)})=\sum_{\ab{i} \in I} b^{(\ell)}_{\ab{i}} N_{\ab{i}}(\bb{\xi}^{(\ell)}) \in \mathcal{S}^{d}_k, \quad \ell=1,2, \ldots, P,
$$
where $b_{\bfm{i}}^{(\ell)} \in \R$ are the spline coefficients of the isogeometric function. 
The associated graph surfaces $\ab{F}^{(\ell)}$ of $w^{(\ell)}$
possess the form
\begin{equation*}
\ab{F}^{(\ell)}(\bb{\xi}^{(\ell)}) 
= \left(\ab{G}^{(\ell)}(\bb{\xi}^{(\ell)}),W^{(\ell)}(\bb{\xi}^{(\ell)}) \right)^T = 
\left( G^{(\ell)}_{1}(\bb{\xi}^{(\ell)}),G^{(\ell)}_{2}(\bb{\xi}^{(\ell)}),W^{(\ell)}(\bb{\xi}^{(\ell)}) \right)^{T}, \quad \ell=1,\ldots, P.
\end{equation*} 

Functions in $V^{(k)}$ can be characterized using the concept of geometric continuity (see \cite{Pe15, KaViJu15}). 
More precisely, an isogeometric 
function $w$ belongs to space $V^{(k)}$ if and only if for all neighboring patches $\Omega^{(\ell)}$ and $\Omega^{(\ell')}$ the two graph surfaces $\ab{F}^{(\ell)}$ and 
$\ab{F}^{(\ell')}$ meet at the common interface with $G^2$ continuity. Therefore, such functions have been called \emph{$C^2$-smooth geometrically continuous isogeometric 
functions} (cf. \cite{KaViJu15}). In \cite{KaplVitrih2016}, the $G^2$-continuity conditions for the two neighboring graphs $\ab{F}^{(\ell)}$ and $\ab{F}^{(\ell')}$, which 
determine linear constraints on the spline coefficients $b_{\bfm{i}}^{(\ell)}$ and $b_{\bfm{i}}^{(\ell')}$, have been derived. Let us shortly recall them.

Without loss of generality we can assume that $
\ab{G}^{(\ell)}(1,{\xi})=
\ab{G}^{(\ell')}(0,{\xi}),\, \xi =\xi_2^{(\ell)}=\xi_2^{(\ell')}\in[0,1].$ Otherwise, suitable linear reparameterizations of the two patches can be applied to fulfill this 
situation. Therefore, an isogeometric function~$w$ is continuous (graphs are $G^0$-smooth) across the common edge if and only if
\begin{equation}\label{eq:480}
W^{(\ell)}(1,{\xi})=
W^{(\ell')}(0,{\xi}). 
\end{equation}
The $C^1$-smoothness of the isogeometric function~$w$ ($G^1$-smoothness of graph surfaces)
is guaranteed if 
\begin{equation}  \label{eq:exampleorder1}
  \det \left(
(\partial_1 \ab{F}^{(\ell)})(1,\xi),  
(\partial_1 \ab{F}^{(\ell')})(0,\xi),
(\partial_2 \ab{F}^{(\ell)})(1,\xi) \right) = 0, \quad
\xi\in[0,1].
\end{equation}
Determinant \eqref{eq:exampleorder1} can be further expressed as
\begin{equation*}  \label{eq:G1conditionAlphaBetaOmega}
\partial_2 W^{(\ell)}(1,\xi) \alpha(\xi ) + \partial_1 W^{(\ell)}(1,\xi) \beta(\xi) - \partial_1 W^{(\ell')}(0,\xi) \gamma(\xi)   = 0, \quad
\xi\in[0,1],
\end{equation*}
where
%\begin{eqnarray}  \label{eq:AlphaBetaGamma}
% \alpha (\xi)  &= 
%    \partial_1 G_1^{(1)}(1,\xi)\;\partial_1 G_2^{(2)}(0,\xi) - \partial_1 G_1^{(2)}(0,\xi)\; \partial_1 %G_2^{(1)}(1,\xi), \nonumber \\
% \beta(\xi) &= 
% \partial_2 G_1^{(1)}(1,\xi)\; \partial_1 G_2^{(2)}(0,\xi) -   \partial_1 G_1^{(2)}(0,\xi)\;
%  \partial_2 G_2^{(1)}(1,\xi) , \\
%\gamma(\xi) & = 
% \partial_1 G_1^{(1)}(1,\xi)\; \partial_2 G_2^{(1)}(1,\xi)  -\partial_2 G_1^{(1)}(1,\xi) \;\partial_1 %G_2^{(1)}(1,\xi). \nonumber 
%\end{eqnarray}
\begin{eqnarray}  \label{eq:AlphaBetaGamma}
 \alpha (\xi)  &= \det \left( \partial_1 \ab{G}^{(\ell)}(1,\xi) , \; \partial_1 \ab{G}^{(\ell')}(0,\xi) \right), 
\nonumber \\
 \beta(\xi) &= 
 \det \left( \partial_1 \ab{G}^{(\ell')}(0,\xi) , \; \partial_2 \ab{G}^{(\ell)}(1,\xi) \right) , \\
\gamma(\xi) & = 
\det \left( \partial_1 \ab{G}^{(\ell)}(1,\xi) , \; \partial_2 \ab{G}^{(\ell)}(1,\xi) \right) . \nonumber 
\end{eqnarray}
$G^{2}$-conditions on both graph surfaces, which finally imply that the isogeometric functions are
 $C^{2}$-smooth, are of the form (cf. \cite{HoLa93,Pe02})
\begin{equation}  \label{eq:exampleorder2}
  \det \left(
  \ab{Z}(\xi),
(\partial_1 \ab{F}^{(\ell)})(1,\xi),
(\partial_2 \ab{F}^{(\ell)})(1,\xi) \right) = 0, \quad
\xi\in[0,1], 
\end{equation}
where
\begin{align*}  \label{eq:Z}
\ab{Z}(\xi) := & (Z_1 (\xi),Z_2 (\xi),Z_3 (\xi))^{T} :=\gamma^2(\xi) (\partial_1^2 \ab{F}^{(\ell')})(0,\xi)  - \nonumber \\
&\left( 
\beta^2(\xi) (\partial_1^2 \ab{F}^{(\ell)})(1,\xi) + 2 \alpha(\xi) \beta(\xi)  
(\partial_1 \partial_2 \ab{F}^{(\ell)})(1,\xi) + \alpha^2(\xi) (\partial_2^2 \ab{F}^{(\ell)})(1,\xi)
\right).
\end{align*}
Conditions \eqref{eq:480}, \eqref{eq:exampleorder1} and \eqref{eq:exampleorder2} determine linear constraints on B-spline coefficients $b_{\bfm{i}}^{(\ell)}$ which can be represented as a homogeneous linear system
\begin{equation}  \label{eq:nullspace}
T^{(k)} \bfm{b} = \bfm{0}, \quad  \bfm{b} = (b_{\bfm{i}}^{(\ell)})_{\bfm{i} \in I, \ell \in \{1,2,\ldots,P\}}.
\end{equation}
Any basis of the nullspace of $T^{(k)}$ now defines a basis of the space $V^{(k)}$.

%%%%%%%%%%%%%%%%%%%%%%%%%%%%%%%%%%%%%%%%%%%%%%%%%%%%%%%%%%%%%%%%%%%%%%%%%%%%%%%%%%%%
\section{Decomposition of the space into the direct sum of subspaces} \label{sec:decompsition}
%%%%%%%%%%%%%%%%%%%%%%%%%%%%%%%%%%%%%%%%%%%%%%%%%%%%%%%%%%%%%%%%%%%%%%%%%%%%%%%%%%%

The decomposition of $V^{(k)}$ into three subspaces will be described and the dimension of the single subspaces will be investigated. Thereby, the 
following theorem will be proved. 

\begin{thm}
Space $V^{(k)}$ can be decomposed into three subspaces, denoted by $V^{(k)}_{\Omega}$, $V^{(k)}_{\Gamma}$ and $V^{(k)}_{\Xi}$, 
such that
\begin{equation} \label{eq:thm:decomposition}
   V^{(k)} = V^{(k)}_{\Omega} \oplus V^{(k)}_{\Gamma} \oplus V^{(k)}_{\Xi},
\end{equation}
where the subspaces are defined in~\eqref{eq:space_Omega}, \eqref{eq:space_Gamma} and \eqref{eq:space_Xi}, respectively.  This implies
\[
\dim V^{(k)} = \dim V^{(k)}_{\Omega} + \dim V^{(k)}_{\Gamma} + \dim V^{(k)}_{\Xi}.
\]
\end{thm}

%%%%%%%%%%%%%%%%%%%%%%%%%%%%%%%%%%%%%%%%%%%%%%%%%%%%%%%%%%%%%%%%%%%%%%%%%%%%%%%%%%%%
\subsection{Index spaces}
%%%%%%%%%%%%%%%%%%%%%%%%%%%%%%%%%%%%%%%%%%%%%%%%%%%%%%%%%%%%%%%%%%%%

Let us first introduce some index spaces of the spline coefficients $\ab{d}_{\ab{i}}^{(\ell)}$ and $b_{\ab{i}}^{(\ell)}$, which will be needed in the following subsections. 
\begin{itemize}
\item For each patch~$\Omega^{(\ell)}$, $\ell=1,2,\ldots,P$, we denote by $\mathcal{I}^{(\ell)}$ the index space 
\[
\mathcal{I}^{(\ell)}= \{\ell \} \times I
\]
of the indices of the spline coefficients $\ab{d}_{\bfm{i}}^{(\ell)}$ of the patch $\Omega^{(\ell)}$, where $I$ is defined in~\eqref{eq:geometry_mapping}. 
The index space $\mathcal{I}$ is the union of all index spaces $\mathcal{I}^{(\ell)}$, $\ell=1,2,\ldots,P$, i.e., 
\[
\mathcal{I} = \bigcup_{\ell=1,2,\ldots,P} \mathcal{I}^{(\ell)}.
\]
\item For each non-boundary edge $\Gamma^{(j)}$, $j=1,2,\ldots, E$, we denote by $\mathcal{I}_{\Gamma^{(j)}} \subset \mathcal{I}$ the space of the indices of the 
spline coefficients $\ab{d}_{\bfm{i}}^{(\ell)}$ of $\Gamma^{(j)} = \Omega^{(\ell)} \cap \Omega^{(\ell')}$ and of the two neighboring columns of spline coefficients 
in $\Omega^{(\ell)}$ and $\Omega^{(\ell')}$, excluding the first three and the last three spline coefficients in each column (see Fig.~\ref{fig:indexSpaces}, 
red coefficients). The index space $\mathcal{I}_{\Gamma} \subset \mathcal{I}$ is defined as the union of all index spaces $\mathcal{I}_{\Gamma^{(j)}}$, $j=1,2,\ldots,E$, 
i.e.,
\[
\mathcal{I}_{\Gamma} = \bigcup_{j=1,2,\ldots,E} \mathcal{I}_{\Gamma^{(j)}}.
\]

\item For each vertex $\ab{v}^{(r)}$, $r=1,2,\ldots,V$, we denote by $\mathcal{I}_{\Xi^{(r)}} \subset \mathcal{I}$ the space of the indices 
$(\ell,\ab{i})=(\ell, (i_{1},i_{2})) \in \mathcal{I}$ of the spline coefficients $\ab{d}_{\ab{i}}^{(\ell)}$, for which the corresponding patch $\Omega^{(\ell)}$ 
contains the vertex $\ab{v}^{(j)}$ and the indices $(i_{1},i_{2})$ satisfy $|i_{1}-i_{1}^{(\ell)}|\leq 2$ and $|i_{2}-i_{2}^{(\ell)}|\leq 2$ when $\ab{v}^{(j)}$ is given 
by $\ab{d}_{(i_{1}^{(\ell)},i_{2}^{(\ell)})}^{(\ell)}$ (see Fig.~\ref{fig:indexSpaces}, cyan coefficients). The index space $\mathcal{I}_{\Xi} \subset \mathcal{I}$ is 
defined as the union of all index spaces $\mathcal{I}_{\Xi^{(r)}}$, $r=1,2,\ldots,V$, i.e.,
\[
\mathcal{I}_{\Xi} = \bigcup_{r=1,2,\ldots,V} \mathcal{I}_{\Xi^{(r)}}.
\]
\end{itemize}

\begin{figure}[htb]\centering
\includegraphics[width=9cm]{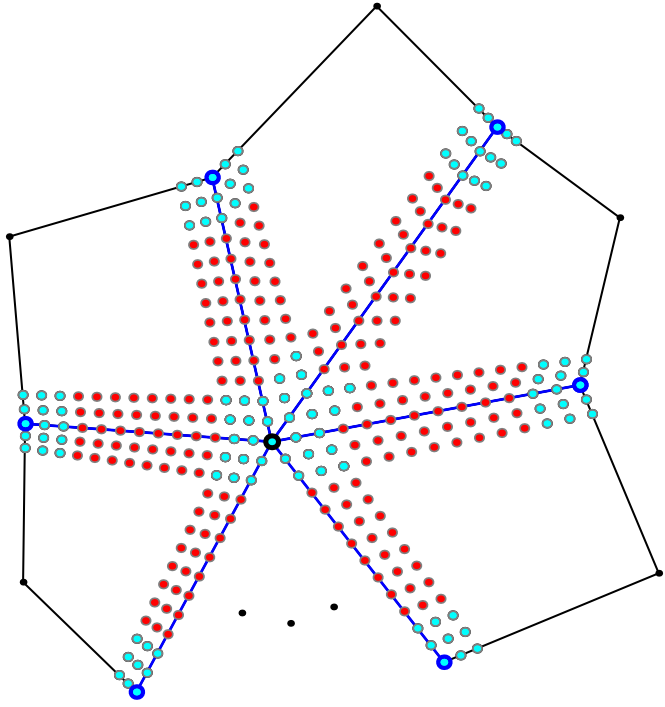}
\caption{A multi-patch domain with one inner vertex (black ring) and some boundary vertices (blue rings). Non-boundary edges are colored in blue and boundary edges in black. 
Spline coefficients corresponding to index spaces $\mathcal{I}_{\Gamma^{(j)}}$ are colored in red for all non-boundary edges $\Gamma^{(j)}$. Furthermore, spline 
coefficients corresponding to index spaces $\mathcal{I}_{\Xi^{(r)}}$ are colored in cyan for all vertices $\bfm{v}^{(r)}$. We draw each spline coefficient 
corresponding to a non-boundary edge~$\Gamma^{(j)}$ only once, although the coefficient occurs twice (and for the inner vertex even more often) in the index 
space $\mathcal{I}$.}
\label{fig:indexSpaces}
\end{figure}

The introduced index spaces for the spline coefficients~$\ab{d}_{\ab{i}}^{(\ell)}$ can be used analogously for the spline coefficients $b_{\ab{i}}^{(\ell)}$. 

%%%%%%%%%%%%%%%%%%%%%%%%%%%%%%%%%%%%%%%%%%%%%%%%%%%%%%%%%%%%%%%%%%%%%%%%%%%%%
\subsection{Patch space}
%%%%%%%%%%%%%%%%%%%%%%%%%%%%%%%%%%%%%%%%%%%%%%%%%%%%%%%%%%%%%%%%%%%%

We consider the space $V^{(k)}_{\Omega} \subset V^{(k)}$, which is defined as
\begin{equation}\label{eq:space_Omega}
V^{(k)}_{\Omega} = \{ w \in V^{(k)} | \; b_{\ab{i}}^{(\ell)}=0 \mbox{ for }(\ell,\ab{i}) \in \mathcal{I}_{\Gamma} \cup \mathcal{I}_{\Xi} \}.
\end{equation}
The space $V_\Omega^{(k)}$ will be called \emph{patch space} and can be decomposed into $P$ 
subspaces $V^{(k)}_{\Omega^{(\ell)}}$, $\ell=1,2,\ldots,P$, such that
\[
V_\Omega^{(k)} = \bigoplus_{\ell \in \{1,2,\ldots,P\} } V_{\Omega^{(\ell)}}^{(k)},
\]
where
\begin{equation} \label{eq:space_Omega_single}
V^{(k)}_{\Omega^{(\ell)}} = \{ w \in V^{(k)} | \; b_{\ab{i}}^{(\ell)}=0 \mbox{ for }(\ell,\ab{i}) \in (\mathcal{I} \setminus \mathcal{I}^{(\ell)}) \cup \mathcal{I}_{\Gamma} \cup \mathcal{I}_{\Xi} \}.
\end{equation}
The dimension of each $V_{\Omega^{(\ell)}}^{(k)}$ depends on the number of boundary edges of $\Omega^{(\ell)}$. 

\begin{lem}
Let $\Omega^{(\ell)}$ be a patch of $\Omega$ with $r_\Gamma$ boundary edges having $r_V$ common vertices. Then
\begin{equation} \label{eq:number_firstkind}
  \dim V_{\Omega^{(\ell)}}^{(k)} =  \begin{cases}
    ((d-2)(k+1)-3)^2, \quad r_\Gamma=0,\\ 
      ((d-2)(k+1)-3) ((d-2)(k+1)),  \quad r_\Gamma=1,\\
       ((d-2)(k+1))^2 - 9,  \quad r_\Gamma=2, \; r_V = 0,\\
        ((d-2)(k+1))^2,  \quad r_\Gamma=2, \; r_V = 1,\\
         ((d-2)(k+1)- 3) ((d-2)(k+1)),  \quad r_\Gamma=3.\\
\end{cases}
\end{equation} 
\end{lem}
\begin{proof}
Since the spline coefficients $b_{\ab{i}}^{(\ell)}$ for $(\ell,\ab{i}) \in \mathcal{I}^{(\ell)} \setminus (\mathcal{I}_{\Gamma} \cup \mathcal{I}_{\Xi})$ are not involved in the linear constraints \eqref{eq:480}, \eqref{eq:exampleorder1} and \eqref{eq:exampleorder2}, we obtain 
\[\dim V_{\Omega^{(\ell)}}^{(k)} = |\mathcal{I}^{(\ell)} \setminus (\mathcal{I}_{\Gamma} \cup \mathcal{I}_{\Xi})|. \] 
\end{proof}

%%%%%%%%%%%%%%%%%%%%%%%%%%%%%%%%%%%%%%%%%%%%%%%%%%%%%%%%%%%%%%%%%%%%%%%%%%%%%%%%%%%%
\subsection{Edge space}
%%%%%%%%%%%%%%%%%%%%%%%%%%%%%%%%%%%%%%%%%%%%%%%%%%%%%%%%%%%%%%%%%%%%

We consider the space $V^{(k)}_{\Gamma} \subset V^{(k)}$, which is defined as 
\begin{equation} \label{eq:space_Gamma}
V^{(k)}_{\Gamma} =\{ w \in V^{(k)} | \; b_{\ab{i}}^{(\ell)}=0 \mbox{ for }(\ell,\ab{i}) \in \mathcal{I} \setminus \mathcal{I}_{\Gamma} \}.
\end{equation}
The space $V^{(k)}_{\Gamma}$ will be called \emph{edge space} and can be decomposed into $E$ subspaces $V^{(k)}_{\Gamma^{(j)}}$, $j=1,2,\ldots,E$, such that
\[
V_\Gamma^{(k)} = \bigoplus_{j \in \{1,2,\ldots,E\} } V_{\Gamma^{(j)}}^{(k)},
\]
where
\begin{equation} \label{eq:space_Gamma_single}
V^{(k)}_{\Gamma^{(j)}} =\{ w \in V^{(k)} | \; b_{\ab{i}}^{(\ell)}=0 \mbox{ for }(\ell,\ab{i}) \in \mathcal{I} \setminus \mathcal{I}_{\Gamma^{(j)}} \}.
\end{equation}
Clearly,
\begin{equation} \label{eq:thmIntersection1}
V^{(k)}_{\Omega} \cap V^{(k)}_{\Gamma}  = \{0\}.
\end{equation}
Any pair of neighboring patches $\Omega^{(\ell)}$ and $\Omega^{(\ell')}$ defines a two-patch subdomain  $\Omega^{(\ell,\ell')} := \Omega^{(\ell)} \cup \Omega^{(\ell')}$ with the common 
edge $\Gamma^{(j)}$. 
The shape of $\Omega^{(\ell,\ell')}$ is determined by six shape points (see Fig.~\ref{fig:two-patchExample}), which will be locally denoted as
\begin{equation}  \label{eq:localVertices_TwoPatch}
\bfm{v}^{(i)}_{\Omega^{(\ell,\ell')} } ,  \quad 
i=0,1,\ldots,5.
\end{equation}
In \cite{KaplVitrih2016} the two-patch case was investigated, where four different configurations of two-patch domains, called Configuration A,B,C and D, were 
introduced (cf. \cite[Section 2]{KaplVitrih2016}) to study the dimension.  The therein presented results can be used to compute the dimension of 
$\dim V_{\Gamma^{(j)}}^{(k)}$, $j \in \{1,2,\ldots ,E \}$.

\begin{figure}[htb]\centering
\includegraphics[width=8cm]{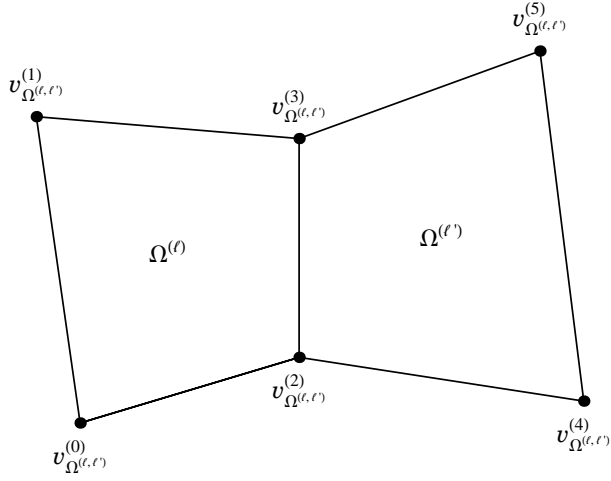}
\caption{Six shape points \eqref{eq:localVertices_TwoPatch} of a two-patch domain $\Omega^{(\ell,\ell')}$.}
\label{fig:two-patchExample}
\end{figure}

\begin{lem}
Let $\Gamma^{(j)} = \Omega^{(\ell)} \cap \Omega^{(\ell')}$ be a non-boundary edge and $\Omega^{(\ell,\ell')} = \Omega^{(\ell)} \cup \Omega^{(\ell')} $ a two-patch subdomain in $\Omega$. Then
\begin{equation} \label{eq:dim1}
\dim V_{\Gamma^{(j)}}^{(k)} = (k+1)\dim V_{\Omega^{(\ell,\ell')}}^{(0)} - 9  \mu - 11 (k+2 -\mu),
\end{equation}
where $\mu \in \{0,1,2,k+2\}$ denotes the number of collinear point triplets 
\begin{equation}  \label{eq:triples}
(1 - \xi_i) \bfm{v}^{(0)}_{\Omega^{(\ell,\ell')}} + \xi_i \bfm{v}^{(1)}_{\Omega^{(\ell,\ell')}}, \quad 
(1 - \xi_i) \bfm{v}^{(2)}_{\Omega^{(\ell,\ell')}} + \xi_i \bfm{v}^{(3)}_{\Omega^{(\ell,\ell')}}, \quad
(1 - \xi_i) \bfm{v}^{(4)}_{\Omega^{(\ell,\ell')}} + \xi_i \bfm{v}^{(5)}_{\Omega^{(\ell,\ell')}}, 
\end{equation}
for $\xi_i = \frac{i}{k+1},  i=0,1,\ldots,k+1,$ and 
\begin{equation*} \label{eq:dim_two_initial}
\dim V_{\Omega^{(\ell,\ell')}}^{(0)}  = \begin{cases}
  3d, \qquad \mbox{if }\Omega^{(\ell,\ell')} \mbox{ is a Configuration A two-patch domain},\\ 
  3d+2, \;\mbox{if }\Omega^{(\ell,\ell')} \mbox{ is a Configuration D two-patch domain},\\
  3d+3,  \;\mbox{if }\Omega^{(\ell,\ell')} \mbox{ is a Configuration B or C two-patch domain}.
\end{cases}
\end{equation*}
\end{lem}

\begin{pf}
Let $r_{1},r_{2} \in \{1,2,\ldots, V\}$ be the two indices of the two boundary vertices of $\Gamma^{(j)}$, i.e., $\ab{v}^{(r_{1})}=\ab{v}^{(2)}_{\Omega^{(\ell,\ell')}}$ and $\ab{v}^{(r_{2})}=\ab{v}^{(3)}_{\Omega^{(\ell,\ell')}}$. Consider the space 
\[
\bar{V}^{(k)} = \{w \in V^{(k)} | \; b_{\ab{i}}^{(\ell)}=0 \mbox{ for } \mathcal{I} \setminus ( (\mathcal{I}_{\Gamma^{(j)}} \cup \mathcal{I}_{\Xi^{(r_{1})}} \cup \mathcal{I}_{\Xi^{(r_{2})}})  \cap (\mathcal{I}^{(\ell)} \cup \mathcal{I}^{(\ell')}) )\}.
\]
Note that functions in $\bar{V}^{(k)}$ are not $C^{2}$-smooth on $\Omega$ but are $C^{2}$-smooth on $\Omega^{(\ell,\ell')}$.
\cite[Theorem 10]{KaplVitrih2016} provides us the dimension of $\bar{V}^{(k)}$, which is equal to
\[
\dim \bar{V}^{(k)} = (k+1) \dim V^{(0)}_{\Omega^{(\ell,\ell')}} -9 \bar{\mu} -11 (k - \bar{\mu}),
\] 
where $\bar{\mu} \in \{0,1,2,k \}$ denotes the number of collinear point triplets~\eqref{eq:triples} for $\xi_i = \frac{i}{k+1},  i=1,2,\ldots,k$. 
The space $V^{(k)}_{\Gamma^{(j)}}$ is equivalent to the space 
\begin{equation}  \label{eq:spaceTwoPatchZeros}
\{ w \in \bar{V}^{(k)} | \; b_{\ab{i}}^{(\ell)}=0 \mbox{ for }(\ell,\ab{i}) \in  \mathcal{I}_{\Xi^{(r_{1})}} \cup \mathcal{I}_{\Xi^{(r_{2})}}\}. 
\end{equation}
Note that functions in the space \eqref{eq:spaceTwoPatchZeros} are $C^{2}$-smooth on the whole domain $\Omega$.
Using \cite[Lemma 7 and 8]{KaplVitrih2016}, we obtain
\begin{equation} \label{eq:dim2}
\dim V^{(k)}_{\Gamma^{(j)}} = \dim \bar{V}^{(k)} - 9 \tilde{\mu} - 11 (2-\tilde{\mu}),
\end{equation} 
where $\tilde{\mu} \in \{0,1,2 \}$ denotes how many of the two point triplets 
\[
\ab{v}^{(0)}_{\Omega^{(\ell,\ell')}}, \; \ab{v}^{(2)}_{\Omega^{(\ell,\ell')}},\;  \ab{v}^{(4)}_{\Omega^{(\ell,\ell')}} \quad \mbox{ and }\quad 
\ab{v}^{(1)}_{\Omega^{(\ell,\ell')}}, \; \ab{v}^{(3)}_{\Omega^{(\ell,\ell')}},\;  \ab{v}^{(5)}_{\Omega^{(\ell,\ell')}}
\]
are collinear. Thus $\mu \in \{0,1,2,k+2 \}$, since $\alpha$ defined in~\eqref{eq:AlphaBetaGamma}, is a quadratic polynomial (compare \cite[Lemma 9]{KaplVitrih2016}). This implies that~\eqref{eq:dim2} is equivalent to~\eqref{eq:dim1}. In addition, the assumption $k\geq 7-d$ ensures that $\dim V^{k}_{\Gamma^{(j)}} \geq 0$. 
\qed
\end{pf}

%%%%%%%%%%%%%%%%%%%%%%%%%%%%%%%%%%%%%%%%%%%%%%%%%%%%%%%%%%%%%%%%%%%%%%%%%%%%%%%%%%%%
\subsection{Vertex space} \label{subsec:vertex_space}
%%%%%%%%%%%%%%%%%%%%%%%%%%%%%%%%%%%%%%%%%%%%%%%%%%%%%%%%%%%%%%%%%%%%

To define the space $V^{(k)}_{\Xi} \subset V^{(k)}$, called {\em vertex space}, we need the concept of finding a minimal determining set for a set of spline 
coefficients with respect to a homogeneous linear system (cf. \cite[][Section 5.6]{LaSch07}).

\begin{defn}
Consider a homogeneous linear system
\begin{equation} \label{eq:system_MDS}
\tilde{T} \tilde{\bfm{b}} = \bfm{0}, \quad  \tilde{\bfm{b}} = (\tilde{b}_{\bfm{i}})_{\bfm{i}}.
\end{equation}
The minimal determining set of the coefficients $\{ \tilde{b}_{\bfm{i}} \}_{\bfm{i}}$ with respect to the system~\eqref{eq:system_MDS} is the smallest subset 
$M \subseteq \{ \tilde{b}_{\bfm{i}} \}_{\bfm{i}}$, such that imposing zero coefficients in $M$ yields vanishing coefficients in 
$M \setminus \{ \tilde{b}_{\bfm{i}} \}_{\bfm{i}}$, too, in order to satisfy~\eqref{eq:system_MDS}.
\end{defn}
Note that in general minimal determining sets are not uniquely determined. For each non-boundary edge~$\Gamma^{(j)}$, $j=1,2,\ldots,E$, we select a minimal determining 
set $M^{(j)}$ of the spline coefficients $\{b_{\ab{i}}^{(\ell)} \}_{(\ell,\ab{i}) \in \mathcal{I}_{\Gamma^{(j)}}}$ with respect to the homogeneous linear 
system~\eqref{eq:nullspace} with the additional constraints $b_{\ab{i}}^{(\ell)}=0$ for $(\ell,\ab{i}) \in \mathcal{I} \setminus \mathcal{I}_{\Gamma^{(j)}}$. 
We denote by $\mathcal{I}_{M^{(j)}} \subset \mathcal{I}$ the space of the indices of the spline coefficients $b_{\ab{i}}^{(\ell)}$ of $M^{(j)}$. 

For each vertex~$\ab{v}^{(r)}$, $r=1,2,\ldots,V$, we define the space $V^{(k)}_{\Xi^{(r)}}$ as
\begin{equation} \label{eq:subspace_singlevertex}
V^{(k)}_{\Xi^{(r)}}= \Big\{w \in V^{(k)} | \; b_{\ab{i}}^{(\ell)} =0 \mbox{ for }(\ell,\ab{i}) \in \mathcal{I} 
\setminus (\mathcal{I}_{\Xi^{(r)}} \cup \bigcup_{j: v^{(r)} \in \Gamma^{(j)}} (\mathcal{I}_{\Gamma^{(j)}}\setminus \mathcal{I}_{M^{(j)}} ) ) \Big\}.
\end{equation}
These spaces are used to define the space $V^{(k)}_{\Xi} \subset V^{(k)}$ as their direct sum, i.e., 
\begin{equation} \label{eq:space_Xi}
V^{(k)}_{\Xi} = \bigoplus_{r \in \{1,2,\ldots,V\}} V_{\Xi^{(r)}}^{(k)}.
\end{equation}
Clearly, we obtain that
\begin{equation}  \label{eq:thmIntersection2}
V^{(k)}_{\Xi} \cap (V^{(k)}_{\Omega} \cup V^{(k)}_{\Gamma} ) = \{0\},
\end{equation}
and assumption $k \geq 7-d$ guarantees that
\begin{equation}  \label{eq:thmUnion}
V^{(k)} = V^{(k)}_{\Omega} \cup V^{(k)}_{\Gamma} \cup V^{(k)}_{\Xi}.
\end{equation}
By \eqref{eq:thmIntersection1}, \eqref{eq:thmIntersection2} and \eqref{eq:thmUnion}, we have proved \eqref{eq:thm:decomposition}.
The dimension of $V^{(k)}_{\Xi^{(r)}}$, $r=1,2,\ldots,V$, will be computed in the following section.

%\textcolor{red}{(Maybe some more information via words -- or maybe in Section 4.1)}

%%%%%%%%%%%%%%%%%%%%%%%%%%%%%%%%%%%%%%%%%%%%%%%%%%%%%%%%%%%%%%%%%%%%%%%%%%%%%%%%%%%%
\section{Dimension of the vertex space} \label{sec:dimension}
%%%%%%%%%%%%%%%%%%%%%%%%%%%%%%%%%%%%%%%%%%%%%%%%%%%%%%%%%%%%%%%%%%%%%%%%%%%%%%%%%%%

The goal of this section is to compute the dimension of $V^{(k)}_{\Xi}$ via the dimension of the single subspaces $V^{(k)}_{\Xi^{(r)}}$, $r=1,2,\ldots,V$.

\subsection{Basic settings} 

For simplicity we will use the following local labelling for each vertex $\ab{v}^{(r)}$, $r=1,2,\ldots,V$. We relabel the considered vertex~$\ab{v}^{(r)}$ 
by $\ab{v}^{(0)}$. We denote the valency of $\ab{v}^{(0)}$ by $\nu$ and define $\nu'$ as $\nu$ if $\ab{v}^{(0)}$ is an inner vertex and as $\nu-1$ if $\ab{v}^{(0)}$ is a 
boundary vertex. Recall that $\nu \geq 3$. In addition, we label the patches $\Omega^{(j)}$, which contain the vertex $\ab{v}^{(0)}$, in counterclockwise order 
as $\Omega^{(\nu-\nu')}, \ldots, \Omega^{(\nu-1)}$, and denote by $\Omega_{\ab{v}^{(0)}}$ the subdomain obtained by these patches, i.e.,
\begin{equation*} \label{eq:v0subdomain}
\Omega_{\bfm{v}^{(0)}} :=  \bigcup_{j=\nu-\nu'}^{\nu-1} \Omega^{(j)},
\end{equation*}
see Fig.~\ref{fig:labeling}. Furthermore, the remaining vertices of $\Omega_{\bfm{v}^{(0)}}$ are labeled in counterclockwise order as 
$\bfm{v}^{(1)}, \bfm{v}^{(2)}$, $\ldots$,$\bfm{v}^{(\nu)}$ %, where $\bfm{v}^{(r)} = (p_r,q_r)^T$ 
(see Fig.~\ref{fig:labeling}). 
Each subdomain $\Omega^{(j)}$, $j>0$, is therefore defined by vertices $\bfm{v}^{(0)}, \bfm{v}^{(j)}, \bfm{v}^{(j+1)}$ and $\widetilde{\bfm{v}}^{(j)}$, where
\begin{equation}  \label{eq:boundaryVertexVal2}
\widetilde{\bfm{v}}^{(j)} := (\tilde{p}_j, \tilde{q}_j)^T
\end{equation}
is either a boundary vertex of valency $2$ or another inner vertex in $\Omega$.
Additionally in case of an inner vertex $\ab{v}^{(0)}$, $\Omega^{(0)}$ is defined by $\bfm{v}^{(0)}, \bfm{v}^{(\nu)}, \bfm{v}^{(1)}$ and $\widetilde{\bfm{v}}^{(\nu)}$.
Spline coefficients $b_{\ab{i}}^{(\ell)}$ with $(\ell,\ab{i}) \in \mathcal{I}_{\Xi^{(0)}}$ will be relabelled in rings around $\bfm{v}^{(0)}$. Let us present this
labelling only for the case of inner vertices, since it works analogously for a boundary vertex only starting from the patch~$\Omega^{(1)}$. For simplicity, we 
consider/draw each spline coefficient corresponding to a common edge~$\Gamma^{(j)}$ only once, although the coefficient could occur more often in the index space. 
More precisely, we denote by $b_0$ the spline coefficient, which corresponds to $\bfm{v}^{(0)}$. Further, we start to label coefficients in each ring around 
$\bfm{v}^{(0)}$ in the middle of $\Omega^{(0)}$. Therefore, the first ring consists of the coefficients $b_1, b_2,\ldots,b_{2\nu}$, while the second one consists of 
$b_{2\nu+1},b_{2\nu+2}, \ldots,b_{6 \nu}$. Let us denote by $\Xi^{(0)}$ these two rings of spline coefficients together with $b_0$.
Moreover, we split $\Xi^{(0)}$ into the union of two disjoint sets $\Xi_1^{(0)}$ and $\Xi_2^{(0)}$ as follows (see Fig.~\ref{fig:labeling}):
\begin{align*} \label{Xi_1_and_2}
 &\Xi_1^{(0)} := \{ b_0 , b_1, b_2,\ldots b_{2\nu}\} \cup \bigcup_{j=0}^{\nu-1}\{ b_{2 \nu+3+4j}\}, \\
 &  \Xi_2^{(0)} := \{ b_{2\nu+1}, \,b_{2\nu+2}, \,b_{6 \nu}  \} \cup \bigcup_{j=0}^{\nu-2} \{ b_{2\nu+4j+4}, \, b_{2\nu+4j+5}, \, b_{2\nu+4j+6} \}.
\end{align*}

\begin{figure}[htb]\centering
\includegraphics[width=10cm]{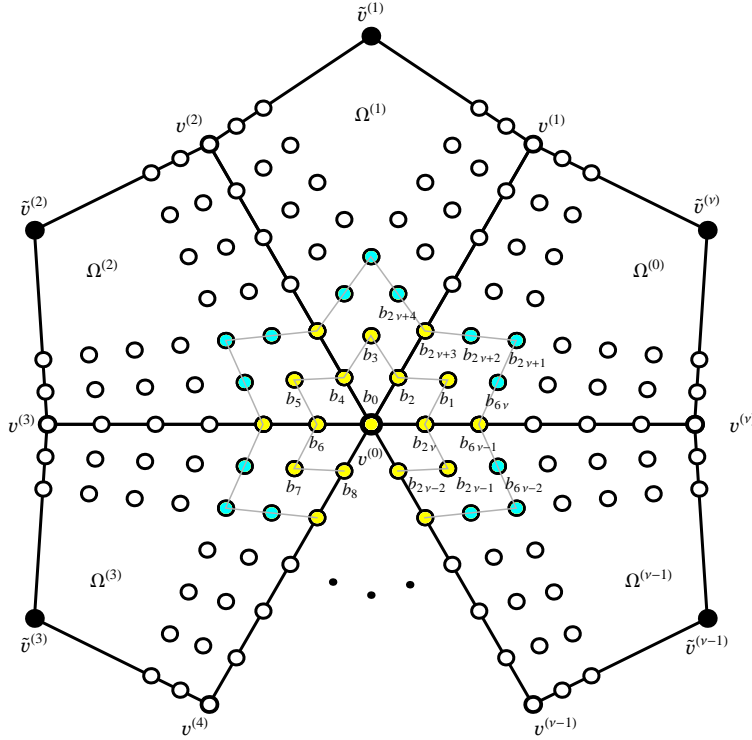}
\caption{Labeling of spline coefficients, vertices and patches in the domain $\Omega_{\bfm{v}^{(0)}}$ with an inner vertex $\bfm{v}^{(0)}$. Coefficients corresponding to 
$\Xi_1^{(0)}$ and $\Xi_2^{(0)}$ are colored in yellow and cyan, respectively. In case of a boundary vertex $\ab{v}^{(0)}$, the labeling works analogously starting 
from $\Omega^{(1)}$.}
\label{fig:labeling}
\end{figure}

In the following, the indices $(j-1,j)$ for the two-patch domains $\Omega^{(j-1,j)}$ are always considered modulo $\nu$.
For each two-patch domain $\Omega^{(j-1,j)}$, $j=\nu-\nu'+1,\ldots,\nu'$, vertex $\bfm{v}^{(0)}$ locally corresponds to the vertex $\bfm{v}_{\Omega^{(j-1, j)}}^{(2)}$.
Furthermore,  we denote by $\Xi^{(j-1,j)}$ the subset of 
$\Xi^{(0)}$ restricted to the two-patch domain $\Omega^{(j-1,j)}$, and by $V^{(k)}_{\Xi^{(j-1,j)}}$ the vertex subspace~\eqref{eq:subspace_singlevertex} at the 
vertex~$\ab{v}^{(0)}$
for the domain $\Omega$ restricted to the two-patch domain $\Omega^{(j-1,j)}$.

\begin{lem} \label{lem:dim_boundary_twopatch}
 Let $\Omega^{(j-1,j)}$ be the two-patch subdomain in $\Omega$. Then 
 \[
  \dim V^{(k)}_{\Xi^{(j-1,j)}} = 
  \begin{cases}
   9, \quad \;\, \mbox{if } \bfm{v}_{\Omega^{(j-1, j)}}^{(0)}, \, \bfm{v}_{\Omega^{(j-1, j)}}^{(2)}, \, \bfm{v}_{\Omega^{(j-1, j)}}^{(4)} 
   \mbox{ are collinear},\\
   11, \quad \mbox{otherwise}.
  \end{cases}
 \]
\end{lem}
\begin{pf}
By \cite[Lemma 7 and 8]{KaplVitrih2016} the number of linearly independent equations, which are formed 
 from the equations of the homogeneous linear system~\eqref{eq:nullspace}, and which are defined on $\Xi^{(j-1,j)}$ only is six if the vertices 
 $\bfm{v}_{\Omega^{(j-1, j)}}^{(0)}, \, \bfm{v}_{\Omega^{(j-1, j)}}^{(2)}, \, \bfm{v}_{\Omega^{(j-1, j)}}^{(4)}$ 
 are collinear, and four otherwise. Using also the fact that $|\Xi^{(j-1,j)}|=15$, the lemma is proved.  
 \qed
\end{pf}

For each two-patch domain $\Omega^{(j-1,j)}$, let us denote by $\ab{e}^{\Xi^{(j-1,j)}}$ the set of the four or six linearly independent equations, which are formed 
from the equations of the homogeneous linear system~\eqref{eq:nullspace}, and which are defined on $\Xi^{(j-1,j)}$ only (see \cite[Equations (18), (28) and (29)]{KaplVitrih2016}). 
Furthemore, let $\ab{e}^{\Xi^{(0)}}$ be the set of the linear equations obtained by the union of all sets 
$\ab{e}^{\Xi^{(j-1,j)}}$, i.e., 
\begin{equation}  \label{eq:eXi0}
\ab{e}^{\Xi^{(0)}} = \bigcup_{j=\nu - \nu' + 1}^{\nu'} \ab{e}^{\Xi^{(j-1,j)}}.
\end{equation}

The following observation provides us a way to prove the dimension of $V^{(k)}_{\Xi^{(0)}}$. Recall that $V^{(k)}_{\Xi^{(0)}}$ is the vertex 
subspace~\eqref{eq:subspace_singlevertex} at the vertex $\ab{v}^{(0)}$.  
\begin{obs}
The dimension of $V^{(k)}_{\Xi^{(0)}}$ is equal to the number of degrees of freedom for the coefficients in $\Xi^{(0)}$ with respect to the system of equations \eqref{eq:eXi0}. 
\end{obs}
We count these degrees of freedom in two steps:
\begin{itemize}
\item[1.] We compute the number of degrees of freedom for the coefficients in $\Xi_1^{(0)}$.
\item[2.] Then we set all coefficients in $\Xi_1^{(0)}$ to zero and compute the number of degrees of freedom for the coefficients in $\Xi_2^{(0)}$.
\end{itemize}
The first step leads to the following result.
\begin{lem}  \label{thm:Xi_1}
   The number of degrees of freedom for $\Xi_1^{(0)}$ with respect to the system of equations \eqref{eq:eXi0} is equal to six.
\end{lem}
\begin{proof}
Let us prove first that the number is at most six. The interpolation of six independent $C^{2}$ data at $\bfm{v}^{(0)}$, i.e., $6$-tuples consisting of the value, the 
two first derivatives and the three second derivatives, uniquely determines all coefficients in $\Xi_1^{(0)}$.
It remains to show that we can indeed always interpolate $C^2$ data at $\bfm{v}^{(0)}$ with a function in $V^{(k)}$, which has zero values for the degrees of 
freedom outside of $\Xi^{(0)}$. The six independent interpolation data uniquely determine a biquadratic polynomial $p$ on $\Omega$, since $p$ is defined by exactly 
six coefficients. Polynomial $p$ is clearly $C^2$ on $\Omega$. Since graph of $p$ is of total degree at most four, $p$ clearly belongs to $V^{(k)}$. Now we take 
its spline representation and choose a function having the same coefficients as $p$ on $\Xi^{(0)}$, while the remaining degrees of freedom outside of $\Xi^{(0)}$ we 
set to zero. The obtained function has the desired properties. It belongs to $V^{(k)}$ and has vanishing degrees of freedom outside of $\Xi^{(0)}$.
\end{proof}
Recall \eqref{eq:shape_points}. To simplify the notation we define
\begin{equation*}  \label{eq:psi}
 \psi_{i,j} := p_i q_j - p_j q_i, \quad i,j \in \{1,2,\ldots,\nu\}.
\end{equation*}
Clearly $\psi_{i,j}= - \psi_{j,i}$. Additionally let $\psi_{0,j} := \psi_{\nu,j}$ and $\psi_{\nu+1,j} := \psi_{1,j}$ for $j=1,2,\ldots,\nu$. 
Note that collinear vertices  $\bfm{v}^{(0)}$, $\bfm{v}^{(i)}$ and $\bfm{v}^{(j)}$ imply $\psi_{i,j} = 0$. 

To perform the second step we first assume that all coefficients in 
$\Xi^{(0)}_{1}$ are set to zero. Then for each two-patch domain $\Omega^{(j-1,j)}$, $j=\nu - \nu' +1,\ldots,\nu'$,
three equations in $\ab{e}^{\Xi^{(j-1,j)}}$ become trivially fulfilled. In the case of non-collinear vertices 
$\bfm{v}_{\Omega^{(j-1, j)}}^{(0)}, \, \bfm{v}_{\Omega^{(j-1, j)}}^{(2)}, \, \bfm{v}_{\Omega^{(j-1, j)}}^{(4)}$
we are left with only one nonzero equation, which we denote by $e^{(j)}$, 
\begin{equation}  \label{eq:support_equation}
  e^{(j)} = e^{(j)}_1 \, b_{2 \nu+4j-4} + e^{(j)}_2 \, b_{2 \nu+4j-2} + e^{(j)}_3 \, b_{2 \nu+4j} + e^{(j)}_4 \, b_{2 \nu+4j+2} = 0, 
\end{equation}
where
\begin{align}  
   & e^{(j)}_1 = \psi_{j,j+1}^3, \quad 
    e^{(j)}_2 = \psi_{j,j+1}^2\, \psi_{j+1,j-1}, \; \nonumber\\[-0.25cm]
    & \label{eq:entries_equation} \\[-0.25cm]
   & e^{(j)}_3 = - \psi_{j-1,j} \,\psi_{j,j+1} \, \psi_{j+1,j-1},\;  \quad
    e^{(j)}_4 = - \psi_{j-1,j}^2 \, \psi_{j,j+1}.\nonumber
\end{align}
In the case that the three vertices 
$\bfm{v}_{\Omega^{(j-1, j)}}^{(0)}, \, \bfm{v}_{\Omega^{(j-1, j)}}^{(2)}, \, \bfm{v}_{\Omega^{(j-1, j)}}^{(4)}$ are collinear, we are left with three nonzero 
equations. These equations are denoted by $\tilde{e}^{(j)}, \bar{e}^{(j)}$, $\hat{e}^{(j)}$ and are of the form
\begin{align}
    & \tilde{e}^{(j)} = \frac{1}{2} \psi_{j,j+1} \, b_{2\nu+4j-2} + \frac{1}{2} \psi_{j-1,j} \, b_{2\nu+4 j} = 0, \nonumber \\
    & \bar{e}^{(j)} = \psi_{j-1,j}^2 \, \psi_{j+1,j} \, b_{2\nu+4j +2} + \psi_{j,j+1}^3 \, b_{2\nu+4j-4} = 0,\qquad   \label{eq:support_equation2} \\[0.1cm]
    & \hat{e}^{(j)} = \left( 3(p_j \tilde{q}_j - \tilde{p}_j q_j) + 7 \psi_{j+1,j}\right) \psi_{j+1,j}^2 \, b_{2\nu+4j-4}  
    -2 \psi_{j+1,j}^3 \, b_{2\nu+4j-3} +  4 \psi_{j+1,j}^3 \, b_{2\nu+4j-2} \nonumber \\
    & \qquad \;  + 2 \psi_{j,j-1}^2 \psi_{j+1,j} \, b_{2\nu+4j+1}+
    C_1^{(j)} \, b_{2\nu+4 j} + C_2^{(j)} \, b_{2\nu+4 j+ 2} = 0,  \nonumber
\end{align}
where $C_1^{(j)}$ and $C_2^{(j)}$ are two constants depending only on shape points of the two-patch domain $\Omega^{(j-1,j)}$.
If $\ab{v}^{(0)}$ is an inner vertex, we have to replace in \eqref{eq:support_equation} and \eqref{eq:support_equation2} $b_{2\nu}$ by $b_{6 \nu}$ and 
$b_{6 \nu+i}$ by $b_{2 \nu+i}$, $i=1,2$, for $j=1$ and $j=\nu$, respectively.

The second step will be completed in the following subsections. Before, we need some notations and assumptions. 
\begin{defn} \label{def:typeVertex}
Vertex $\bfm{v}^{(0)}$ is called a \emph{type} $\rho$, $\rho \in \N_0$, vertex if 
vertices 
$\bfm{v}_{\Omega^{(j-1, j)}}^{(0)}, \, \bfm{v}_{\Omega^{(j-1, j)}}^{(2)},$ 
and $\bfm{v}_{\Omega^{(j-1, j)}}^{(4)}$
are collinear exactly for $\rho$ two-patch domains $\Omega^{(j-1,j)} \subseteq \Omega_{\bfm{v}^{(0)}}$. 
\end{defn}
\begin{rem}  \label{rem:types}
By geometry reasons there exist only type 0, type 1, type 2 and for $\nu=4$ also type 4 vertices. 
Furthermore, for $\nu > 4$ in the case of a type 2 vertex, the two two-patch domains which define the type 2 vertex need to have one patch in common. 
Compare Fig.~\ref{fig:examples_types0}, Fig.~\ref{fig:examples_types12} and Fig.~\ref{fig:examples_typesVal4}.
\end{rem}
\begin{figure}[htb]\centering
\begin{minipage}[]{5.5cm}
\includegraphics[width=4cm]{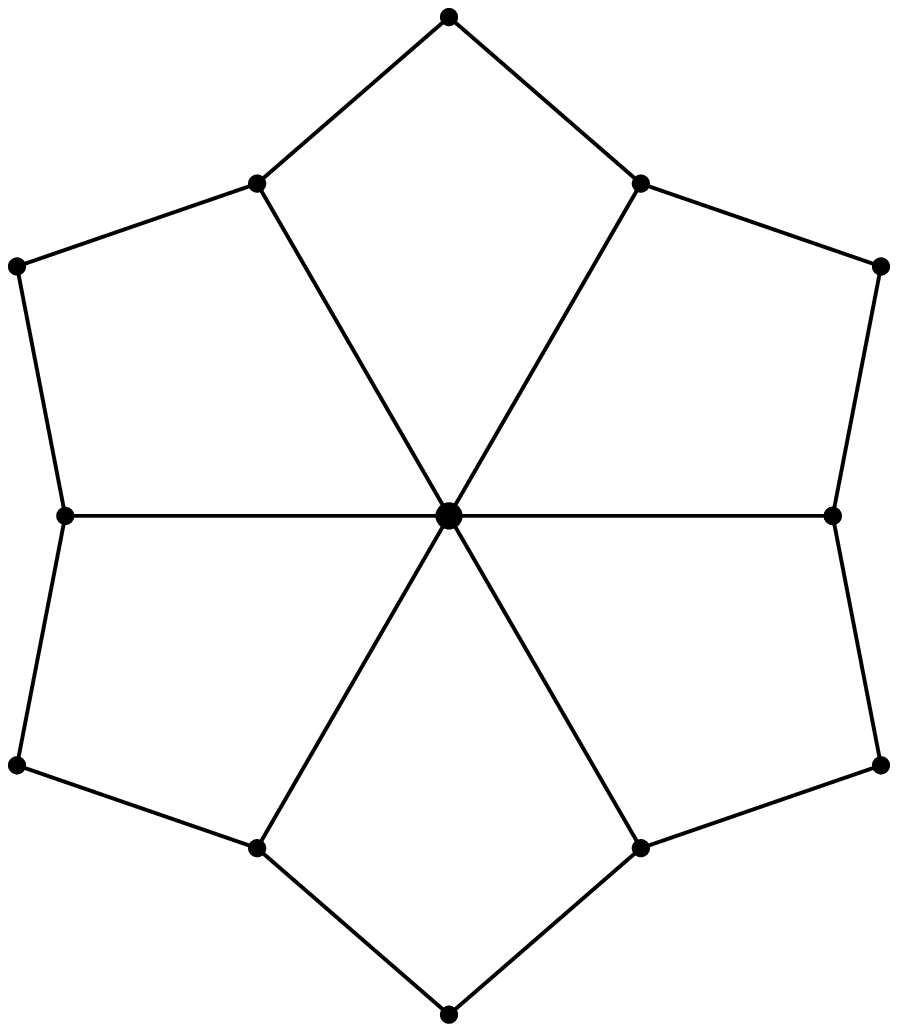}
\end{minipage}
\hskip1em
\begin{minipage}[]{5.5cm}
\includegraphics[width=4cm]{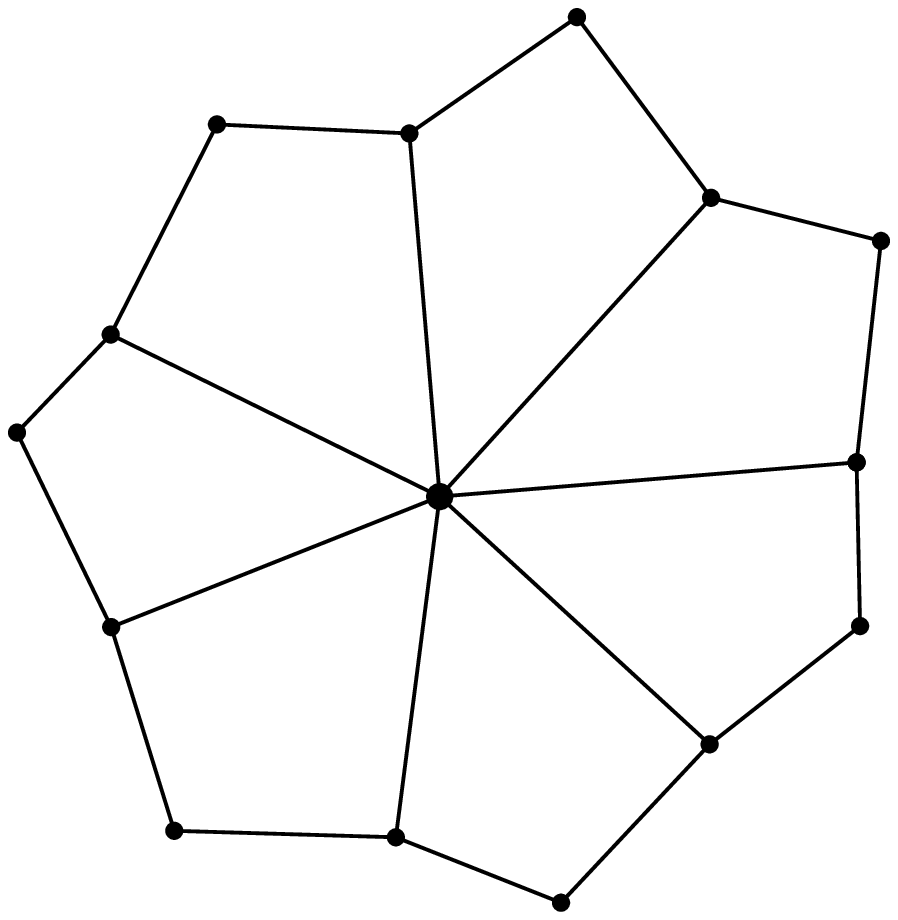}
\end{minipage}
\caption{Two examples for type 0 inner vertices.}
\label{fig:examples_types0}
\end{figure}
\begin{figure}[htb]\centering
\begin{minipage}[]{5.5cm}
\includegraphics[width=4cm]{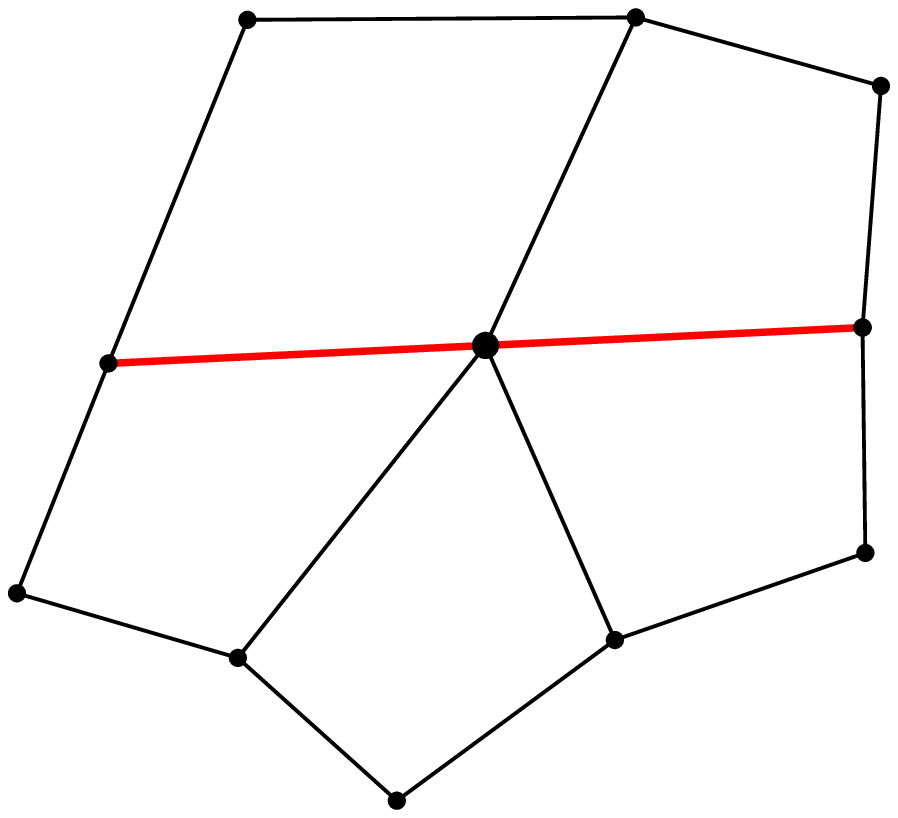}
\end{minipage}
\hskip0.5em
\begin{minipage}[]{5.5cm}
\includegraphics[width=4cm]{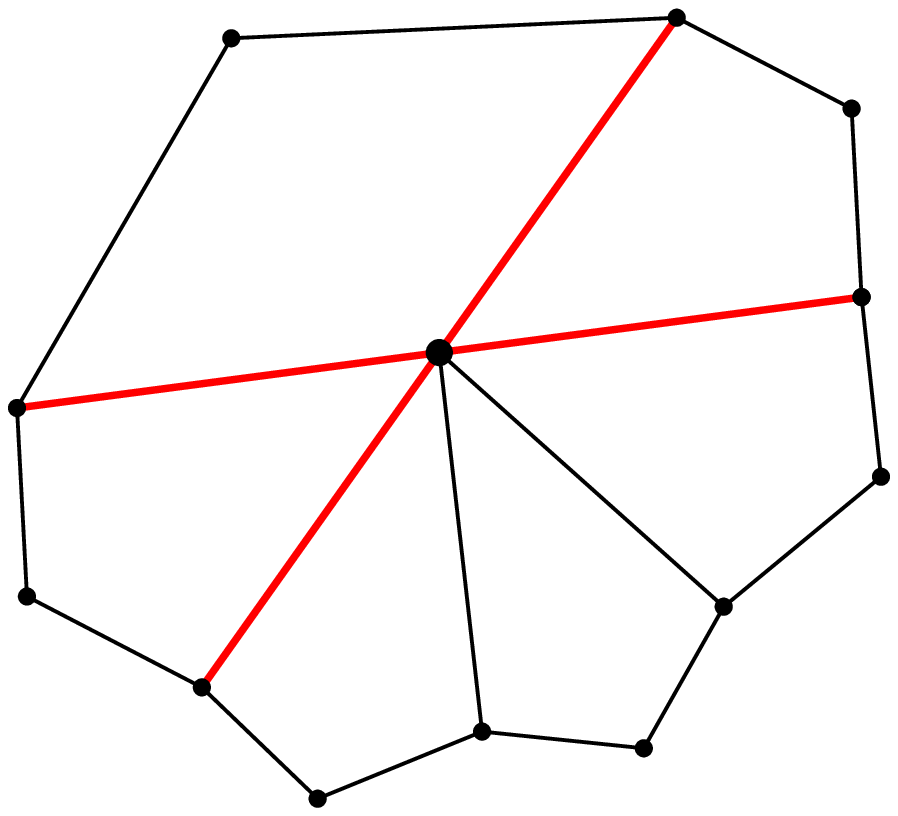}
\end{minipage}
\caption{Examples for type 1 (left) and type 2 (right) inner vertices.}
\label{fig:examples_types12}
\end{figure}
\begin{figure}[htb]\centering
\begin{minipage}[]{5.5cm}
\includegraphics[width=4cm]{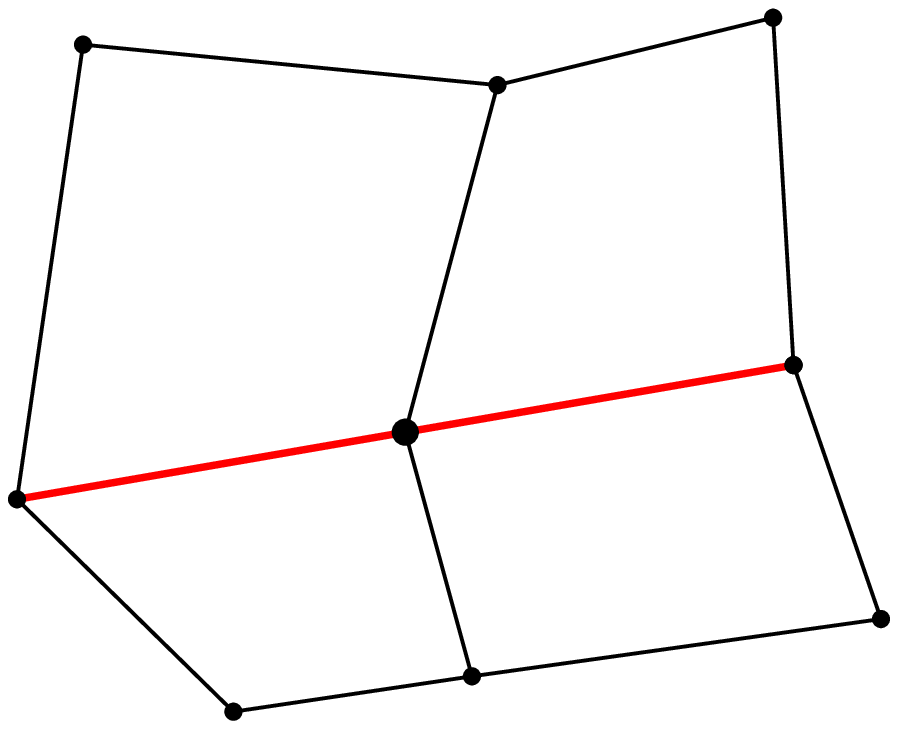}
\end{minipage}
\hskip0.5em
\begin{minipage}[]{5.5cm}
\includegraphics[width=4cm]{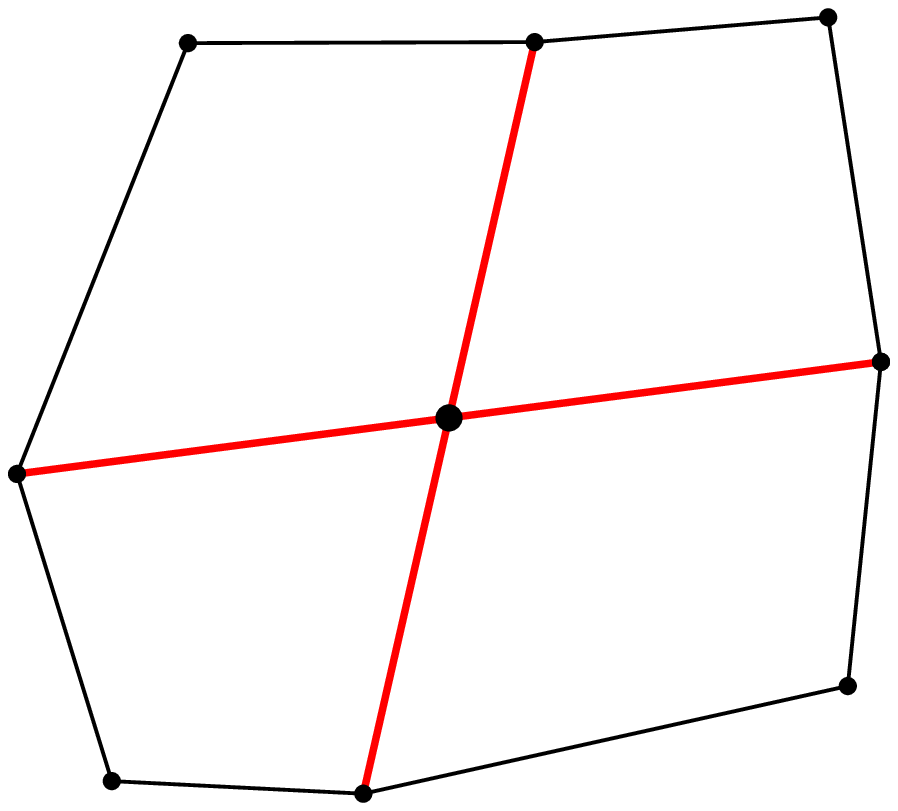}
\end{minipage}
\caption{Type 2 (left) and type 4 (right) inner vertices for valency $\nu=4$.}
\label{fig:examples_typesVal4}
\end{figure}

We will show that  $\dim V^{(k)}_{\Xi^{(0)}}$ depends on the valency $\nu$ and the type $\rho \in \{0,1,2,4\}$ of the vertex $\bfm{v}^{(0)}$. Therefore, we will use the 
notations $N_{\nu}^{{\rm (\rho)}}$ and $\widetilde{N}_{\nu}^{{\rm (\rho)}}$ to denote the dimension $\dim V^{(k)}_{\Xi^{(0)}}$ for inner and boundary vertex $\bfm{v}^{(0)}$, respectively. 
Without loss of generality we will assume that
$
\bfm{v}^{(0)} = (0,0)^T.
$ 

\subsection{Type 0 vertices}

Let $\bfm{v}^{(0)}$ be a type 0 vertex. 
Then 
$
\bfm{e}^{\Xi^{(0)}} = \{e^{(\nu - \nu' +1)},\ldots,e^{(\nu')} \}. 
$
Lemma~\ref{thm:Xi_1} and equality $| \Xi_2^{(0)} | = 3 \nu'$ yield
$
N_{\nu}^{(0)} \geq 6+3 \nu - \nu = 2\nu+6
$ 
and
$
\widetilde{N}_{\nu}^{(0)} \geq 6+3 (\nu-1) - (\nu-2) = 2\nu+5.
$ 
The lower bound is reached when all equations in $\bfm{e}^{\Xi^{(0)}}$ are linearly independent. 
In case of a boundary vertex~$\ab{v}^{(0)}$, all 
equations in $\bfm{e}^{\Xi^{(0)}}$ are clearly linearly independent and therefore $\widetilde{N}_{\nu}^{(0)} = 2\nu+5$. In case of an inner vertex only equations in 
$\bfm{e}^{\Xi^{(0)}} \backslash \{e^{(\nu)}\}$ are clearly linearly independent. Depending on the valency $\nu$, the equation $e^{(\nu)}$ is linearly dependent or 
independent with the equations in $\bfm{e}^{\Xi^{(0)}} \backslash \{e^{(\nu)}\}$, and therefore $2\nu+6 \leq  N_{\nu}^{(0)} \leq 2\nu+7$. We need the following lemma to 
compute $N_{\nu}^{(0)}$.
\begin{lem}  \label{lemma:Dj}
Let 
\begin{equation} \label{eq:Djnu}
D_j^\nu :=   \begin{bmatrix}
    e^{(j)}_3 &e^{(j+1)}_1   \\[0.2cm]
    e^{(j)}_4 & e^{(j+1)}_2
\end{bmatrix}, \quad j=1,2,\ldots,\nu-1.
\end{equation}
Then
$$
\det D_j^\nu = \psi_{j-1,j} \,  \psi_{j,j+1}^2 \, \psi_{j+1,j+2}^2 \,   \psi_{j+2,j-1}. 
$$
\end{lem}
\begin{pf}
    By using \eqref{eq:entries_equation} we get
    \begin{align}
       \det D_j^\nu =& -  \psi_{j-1,j}  \psi_{j,j+1}  \psi_{j+1,j-1}  \psi_{j+1,j+2}^2  \psi_{j+2,j} +  \psi_{j-1,j}^2 \psi_{j,j+1} \psi_{j+1,j+2}^3 \nonumber \\
       & =  \psi_{j+1,j+2}^2 \,  \psi_{j-1,j}\,  \psi_{j,j+1} \left(  - \psi_{j+1,j-1} \,  \psi_{j+2,j} + 
        \psi_{j-1,j} \,  \psi_{j+1,j+2}\right) \nonumber \\
        &  = \psi_{j-1,j} \,  \psi_{j,j+1}^2 \, \psi_{j+1,j+2}^2 \,   \psi_{j+2,j-1}. \nonumber
    \end{align}\qed
\end{pf}

\begin{lem}  \label{thm:type0}
  Let $\bfm{v}^{(0)}$ be a type 0 vertex of valency $\nu \geq 3$. Then
  \begin{equation*}
     N_\nu^{(0)} = \begin{cases}
  13, \quad \quad \nu=3,\\ 
  2\nu+6,  \; \nu \geq 4,
\end{cases}
\quad {\rm and} \qquad
     \widetilde{N}_\nu^{(0)} = 
  2\nu+5.
  \end{equation*}
\end{lem}
\begin{pf}
It has been already shown above that $\widetilde{N}_\nu^{(0)}=2\nu+5$.
To compute $N_\nu^{(0)}$, let $\bfm{v}^{(0)}$ be an inner vertex. For $\nu=3$ a straightforward computation gives
$$
  e^{(3)} = \frac{\psi_{1,3} \,\psi_{2,3}}{\psi_{1,2}^2} e^{(1)} + \frac{\psi_{1,3}^2}{\psi_{1,2}\, \psi_{2,3}} e^{(2)}.
$$
Therefore, $N_3^{(0)} = 6+3\cdot3-(3-1) = 13$.
Let now $\nu \geq 4$.
We have to prove that equations $e^{(1)},e^{(2)}, \ldots,e^{(\nu)}$ are linearly independent. Let us construct a matrix $A_\nu^{(0)} \in \R^{2\nu \times \nu}$, 
where the $j$-th column is represented by coefficients \eqref{eq:entries_equation} of $e^{(j)}$ and rows correspond to spline coefficients 
$b_{6\nu}, b_{2\nu+2}, b_{2\nu+4}, \ldots ,b_{6\nu-2}$. Matrix $A_\nu^{(0)} $ is of the form
$$
\left[\begin{array}{ccccccccccccc}
    e^{(1)}_1 &  e^{(1)}_2 &  e^{(1)}_3 &  e^{(1)}_4 & 0  & 0 & \cdots &  0  & 0 &  0  & 0 & 0  & 0  \\[0.1cm]
     0 & 0 &  e^{(2)}_1 & e^{(2)}_2 & e^{(2)}_3  & e^{(2)}_4 & \cdots &  0  & 0 &  0  & 0 & 0  & 0  \\[0.1cm]
      \vdots &  \vdots &   \vdots &  \vdots &  \vdots  &  \vdots &  \ddots &  \vdots  &  \vdots &   \vdots  &  \vdots &  \vdots  &  \vdots  \\[0.1cm]
      0  & 0 &  0  & 0 & 0  & 0 & \cdots & e^{(\nu-2)}_1 & e^{(\nu-2)}_2 & e^{(\nu-2)}_3 & e^{(\nu-2)}_4 & 0 & 0 \\[0.1cm]
      0  & 0 &  0  & 0 & 0  & 0 & \cdots & 0 & 0 & e^{(\nu-1)}_1 & e^{(\nu-1)}_2 & e^{(\nu-1)}_3 & e^{(\nu-1)}_4 \\[0.1cm] 
      e^{(\nu)}_3 & e^{(\nu)}_4 & 0  & 0 &  0  & 0 & \cdots & 0 & 0 & 0 & 0 & e^{(\nu)}_1 & e^{(\nu)}_2\\[0.1cm] 
\end{array}
\right]^T.
$$
We have to find a nonsingular $\nu \times \nu$ submatrix in $A_\nu^{(0)}$.
   If we take the rows
   $
\bigcup_{i=2}^{\nu} \{2 i\} \cup \{2\nu-1\}, 
$
we get an almost upper triangular $\nu \times \nu$ submatrix $\widetilde{A}_{\nu}^{(0)}$. By Lemma~\ref{lemma:Dj} we can easily see that
   \begin{equation}  \label{eq:detA0}
      \det \widetilde{A}_{\nu}^{(0)} = (-1)^{\nu-2}  \, \psi_{1,\nu-2}  \, \psi_{1,\nu}^4 \, \psi_{\nu-2,\nu-1}^2 \, \psi_{\nu-1,\nu}^2 \, \prod_{\ell=1}^{\nu-3}  \psi_{\ell,\ell+1}^3.
      \end{equation}
      Since $\bfm{v}^{(0)}$ is a type 0 vertex, $\det \widetilde{A}_{\nu}^{(0)} \neq 0$.\qed
\end{pf}

\subsection{Type 1 vertices}

Let $\bfm{v}^{(0)}$ be a type 1 vertex. If $\ab{v}^{(0)}$ is an inner vertex, we assume without loss of generality that 
\begin{equation}  \label{eq:psi32nu-1}
\psi_{2,\nu}=0,
\end{equation}
which corresponds to the fact that $\Omega^{(0,1)}$ is the only two-patch domain for which vertices  
$\bfm{v}_{\Omega^{(0, 1)}}^{(0)},  \bfm{v}_{\Omega^{(0, 1)}}^{(2)}$ and $\bfm{v}_{\Omega^{(0, 1)}}^{(4)}$
are collinear. In the case of a boundary vertex $\bfm{v}^{(0)}$, we assume that $\Omega^{(1,2)}$ is such a two-patch domain.
In order to have a non-degenerate case, inner vertices of valency $3$ cannot be type 1 vertices. Moreover, inner vertex $\bfm{v}^{(0)}$ of valency $4$ cannot be a 
type 1 vertex, too, since then also for the two-patch domain $\Omega^{(2,3)}$ the vertices  
$\bfm{v}_{\Omega^{(2, 3)}}^{(0)},  \bfm{v}_{\Omega^{(2, 3)}}^{(2)}$ and $\bfm{v}_{\Omega^{(2, 3)}}^{(4)}$
would be collinear and hence $\bfm{v}^{(0)}$ would be a type 2 vertex (see Fig.~\ref{fig:examples_typesVal4}, left). Note that 
boundary vertices of type 1 exist also 
for valencies $3$ and $4$.
The set of equations  $\bfm{e}^{\Xi^{(0)}}$ is now of the form
$$
\bfm{e}^{\Xi^{(0)}} = \{ \tilde{e}^{(\nu-\nu'+1)}, \bar{e}^{(\nu-\nu'+1)}, \hat{e}^{(\nu-\nu'+1)}, e^{(\nu-\nu'+2)}, e^{(\nu-\nu'+3)}, \ldots, e^{(\nu')} \},
$$
which implies $
N_{\nu}^{(1)} \geq  6+3 \nu - (\nu+2) = 2\nu + 4$
and
$
\widetilde{N}_{\nu}^{(1)} \geq 6+3 (\nu-1) - \nu = 2\nu+3.
$ 
\begin{lem}  \label{thm:type1}
   Let $\bfm{v}^{(0)}$ be a type 1 vertex of valency $\nu$. Then 
   $$
   N_{\nu}^{(1)} = 2\nu+4, \; \,\nu \geq 5, \qquad {\rm and} \qquad \widetilde{N}_{\nu}^{(1)} = 2\nu +3, \; \, \nu \geq 3.
   $$
\end{lem}
\begin{pf}
Let first $\bfm{v}^{(0)}$ be an inner vertex. 
   Equation $\hat{e}^{(1)}$ is the only equation in  $\bfm{e}^{\Xi^{(0)}}$
   which support includes $b_{2\nu+1}$ and is therefore (since $\psi_{2,1} \neq 0$) clearly linearly independent with all other equations in $\bfm{e}^{\Xi^{(0)}}$. 
   It remains to prove that also all equations in $\bfm{e}^{\Xi^{(0)}} \backslash \{\hat{e}^{(1)}\}$
   are linearly independent among them. Similarly as in Lemma~\ref{thm:type0} we construct a matrix $A_\nu^{(1)} \in \R^{2 \nu \times (\nu+1)}$, where 
   columns $3,4,\ldots,\nu+1$ are the same as the columns $2,3,\ldots,\nu$ in $A_{\nu}^{(0)}$, while the first two columns correspond to $\tilde{e}^{(1)}$ and 
   $\bar{e}^{(1)}$. The rows again correspond to $b_{6\nu}, b_{2\nu+2}, b_{2\nu+4}, \ldots ,b_{6\nu-2}$.
 By selecting the rows 
 $
\bigcup_{i=2}^{\nu} \{2 i-1\} \cup \{4,2\nu\}, \;
$
we obtain an almost upper triangular $(\nu+1) \times (\nu+1)$ submatrix $\widetilde{A}_{\nu}^{(1)}$. It is straightforward to see that
\begin{equation}  \label{eq:detA1}
   \det \widetilde{A}_{\nu}^{(1)} = \frac{1}{2} \det \widetilde{A}_{\nu}^{(0)} \frac{1}{\psi_{2,1} \psi_{1,\nu}^2}  \psi_{1,\nu}^3\,\psi_{1,2} = -\frac{1}{2} \psi_{1,\nu} \det \widetilde{A}_{\nu}^{(0)}.
\end{equation}  
The last expression in \eqref{eq:detA1} does not involve the term $\psi_{2,\nu}$, which implies that $\det \widetilde{A}_{\nu}^{(1)} \neq 0$.

Let now $\bfm{v}^{(0)}$ be a boundary vertex and $\Omega_{\bfm{v}^{(0)}}$ consists of $(\nu-2)$ two-patch domains. 
For the case $\nu=3$ Lemma~\ref{lem:dim_boundary_twopatch} proves the result. For $\nu \geq 4$ the same arguments as for the inner vertex show linear independency of equations in $\bfm{e}^{\Xi^{(0)}}$  and thus $\widetilde{N}_{\nu}^{(1)} = 6 + 3 (\nu-1) - (3+(\nu-3)) = 2 \nu+3$. 
 \qed
\end{pf}

\subsection{Type 2 vertices}

Let $\bfm{v}^{(0)}$ be a type 2 vertex. Clearly the valency $\nu$ of $\bfm{v}^{(0)}$ is $\nu \geq 4$. If $\ab{v}^{(0)}$ is an inner vertex,
we assume without loss of generality that \eqref{eq:psi32nu-1} holds, which corresponds to the fact that $\Omega^{(0,1)}$ is one of the two two-patch subdomains for 
which the three vertices $\bfm{v}_{\Omega^{(0, 1)}}^{(0)}, \, \bfm{v}_{\Omega^{(0, 1)}}^{(2)}$ and $\bfm{v}_{\Omega^{(0, 1)}}^{(4)}$ 
are collinear. In the case of a boundary vertex $\bfm{v}^{(0)}$, we assume that $\Omega^{(1,2)}$ is such a two-patch domain.

\begin{lem}  \label{thm:type2}
  Let $\bfm{v}^{(0)}$ be a type 2 vertex of valency $\nu \geq 4$. Then
  \begin{equation}  \label{eq:type2generalnu}
     N_\nu^{(2)} = \begin{cases}
  2\nu+3,  \; \nu \in \{ 4,5\},\\
  2\nu+2,  \; \nu \geq  6,
\end{cases}
\quad {\rm and} \qquad
\widetilde{N}_\nu^{(2)} = 2 \nu+1.
  \end{equation}
\end{lem}

\begin{pf}
Let first $\bfm{v}^{(0)}$ be an inner vertex. We prove formula \eqref{eq:type2generalnu} for $\nu \in \{4,5\}$ in Appendix A.
Let now $\nu \geq 6$. It is straightforward to see that $\Omega^{(1,2)}$ or $\Omega^{(\nu-1,\nu)}$ has to be the second two-patch domain 
which makes $\bfm{v}^{(0)}$ to be type $2$ vertex (compare Remark~\ref{rem:types}).
Without loss of generality we assume that it is $\Omega^{(1,2)}$. Therefore, 
\begin{equation} \label{eq:psiForTwoLines}
\psi_{1,3}=0 \quad {\rm and} \quad \psi_{2,\nu} = 0.
\end{equation}
The set $\bfm{e}^{\Xi^{(0)}}$ consists of equations 
\[ \bfm{e}^{\Xi^{(0)}} = \{\tilde{e}^{(1)}, \bar{e}^{(1)}, \hat{e}^{(1)}, \tilde{e}^{(2)}, \bar{e}^{(2)}, 
\hat{e}^{(2)}, e^{(3)}, e^{(4)}, \ldots, e^{(\nu)}\}, 
\]
where the equations $\hat{e}^{(1)}$ and $\hat{e}^{(2)}$ are clearly linearly independent with all the others.
Moreover, it is straightforward to see that all equations in $\bfm{e}^{\Xi^{(0)}} \backslash \{\hat{e}^{(1)},\hat{e}^{(2)},e^{(\nu)}\}$ 
are linearly independent.  
Let us construct a matrix $A_\nu^{(2)} \in \R^{(2 \nu)\times (\nu+2)}$, where each column corresponds to one of the equations in 
$\bfm{e}^{\Xi^{(0)}} \backslash \{\hat{e}^{(1)},\hat{e}^{(2)}\}$
and rows correspond to  $b_{6\nu}, b_{2\nu+2}, b_{2\nu+4}$, $\ldots$, $b_{6\nu-2}$.
By selecting the rows with indices
$
\bigcup_{i=2}^{\nu} \{2 i\} \cup \{3,5,2\nu-1\}, \;
$
we obtain an almost upper triangular $(\nu+2) \times (\nu+2)$ submatrix $\widetilde{A}_{\nu}^{(2)}$, for which it is easy to see that
\begin{equation*}  \label{eq:detA2}
   \det \widetilde{A}_{\nu}^{(2)} = - \frac{1}{4} \psi_{1,2} \psi_{1,\nu}\det \widetilde{A}_{\nu}^{(0)} = \frac{1}{2} \psi_{1,2} \det \widetilde{A}_{\nu}^{(1)}.
\end{equation*}  
By \eqref{eq:detA0} and \eqref{eq:psiForTwoLines} it follows that $\det \widetilde{A}_{\nu}^{(2)}  \neq 0$.
Note that for $\nu=5$ expression \eqref{eq:detA0} involves the term $\psi_{1,3}=0$.

Let now $\bfm{v}^{(0)}$ be a boundary vertex of type $2$ and $\Omega_{\bfm{v}^{(0)}}$ consists of $(\nu-2)$ two-patch domains. 
Clearly, 
$\Omega^{(2,3)}$ is the second two-patch domain for which the vertices $\bfm{v}_{\Omega^{(2, 3)}}^{(0)}, \, \bfm{v}_{\Omega^{(2, 3)}}^{(2)}$ and 
$\bfm{v}_{\Omega^{(2, 3)}}^{(4)}$ are collinear. Then the set $\bfm{e}^{\Xi^{(0)}}$ consists of equations 
\[
\bfm{e}^{\Xi^{(0)}} = \{\tilde{e}^{(2)}, \bar{e}^{(2)}, \hat{e}^{(2)}, 
\tilde{e}^{(3)}, \bar{e}^{(3)}, \hat{e}^{(3)}, e^{(4)}, \ldots, e^{(\nu-1)}\}. 
\] Since all equations in $\bfm{e}^{\Xi^{(0)}}$ are linearly independent it follows that
$\widetilde{N}_\nu^{(2)} = 6+3(\nu-1) - (6+(\nu-4)) = 2\nu+1$.
\qed
\end{pf}

\subsection{Type 4 vertices}

Recall that only inner vertices of valency $\nu=4$ can be type $4$ vertices (compare  Remark~\ref{rem:types} and Fig.~\ref{fig:examples_typesVal4} 
(right)). We have equations 
\begin{equation}  \label{eq:12equations}
\bfm{e}^{\Xi^{(0)}}= \{\tilde{e}^{(1)}, \bar{e}^{(1)}, \hat{e}^{(1)}, \tilde{e}^{(2)}, \bar{e}^{(2)}, \hat{e}^{(2)}, \tilde{e}^{(3)}, \bar{e}^{(3)}, \hat{e}^{(3)},  
\tilde{e}^{(4)}, \bar{e}^{(4)}, \hat{e}^{(4)} \},
\end{equation}
which form columns of a matrix $A_4^{(4)}$.
By considering their supports, the first nine equations are clearly linearly independent. Moreover, the last three equations are all dependent by the first nine, 
which is shown in the next lemma.
\begin{lem} \label{lem:type4}
 Let $\bfm{v}^{(0)}$ be a type 4 inner vertex of valency $4$. Among the equations in \eqref{eq:12equations} only nine are linearly independent, 
 therefore $ N_4^{(4)} =  9$.  The equations $\bar{e}^{(4)}$, $\tilde{e}^{(4)}$ and $\hat{e}^{(4)}$  can be expressed as 
\begin{align}
   & \bar{e}^{(4)} = \alpha_1 \, \tilde{e}^{(1)} + \alpha_2 \, \bar{e}^{(2)} + \alpha_3 \, \tilde{e}^{(3)}, \label{eq:val4type4_1_2} \qquad 
   \tilde{e}^{(4)} = \beta_1 \, \hat{e}^{(1)} + \beta_2 \, \tilde{e}^{(2)} + \beta_3 \, \bar{e}^{(3)},  \\
  &  \hat{e}^{(4)}  = \gamma_1 \, \tilde{e}^{(1)} +  \gamma_2 \, \bar{e}^{(1)} +  \gamma_3 \, \hat{e}^{(1)} +  \gamma_4 \,  \tilde{e}^{(2)} +
   \gamma_5 \,  \bar{e}^{(2)} +  \gamma_6 \, \hat{e}^{(2)} +  \gamma_7 \, \tilde{e}^{(3)} +  \gamma_8 \,  \bar{e}^{(3)} +  \gamma_9 \, \hat{e}^{(3)}, \label{eq:val4type4_3}
\end{align}
for some coefficients $\alpha_i, \beta_i$, $i=1,2,3$, and $\gamma_j$, $j=1,2,\ldots,9$, which depend on vertices $\bfm{v}^{(r)}$ and $\widetilde{\bfm{v}}^{(r)}$, $r=1,2,3,4$, only. 
\end{lem}
\begin{pf}
  See Appendix B. \qed
\end{pf}

\subsection{Summary}
Let us summarize the results from the previous subsections.
\begin{thm}
Let $\ab{v}^{(r)}$ be a vertex of valency~$\nu \geq 3$ and of type~$\rho \in \{0,1,2,4\}$ and let $V^{(k)}_{\Xi^{(r)}}$ be the space defined in \eqref{eq:subspace_singlevertex}. 
If $\ab{v}^{(r)}$ is an inner vertex then
\[
\dim V^{(k)}_{\Xi^{(r)}} = N_\nu^{(\rho)} = 
\begin{cases}
6 + 2(\nu-\rho) , \quad  \rho \in \{0,1,2\} \;\, {\rm and }\;\, \nu \geq 4+\rho, \\
7+2 (\nu-\rho) , \quad (\nu,\rho) \in \{ (3,0), (4,2), (5,2)\}, \\
9 , \qquad \qquad  \;\;\,\quad (\nu,\rho) =(4,4).
\end{cases}
\]
For a boundary vertex $\ab{v}^{(r)}$ it holds
\[
 \dim V^{(k)}_{\Xi^{(r)}} = \widetilde{N}_\nu^{(\rho)} =   5 + 2(\nu-\rho) , \quad  \nu \geq \max \{3,2+\rho\}.
\]  
\end{thm}
\begin{pf}
 Consider Lemma~\ref{thm:type0} 
 -~\ref{lem:type4}.
 \qed
\end{pf}

%%%%%%%%%%%%%%%%%%%%%%%%%%%%%%%%%%%%%%%%%%%%%%%%%%%%%%%%%%%%%%%%%%%%%%%
\section{Numerical examples}   \label{sec:Examples}
%%%%%%%%%%%%%%%%%%%%%%%%%%%%%%%%%%%%%%%%%%%%%%%%%%%%%%%%%%%%%%%%%%%%%%%%

In this section, we will explore the potential of the space of $C^2$-smooth geometrically continuous isogeometric functions by performing $L^{2}$ approximation and 
by numerically solving the triharmonic equation on two different configurations of bilinear multi-patch domains. First, let us describe a basis construction of the 
space of $C^2$-smooth geometrically continuous isogeometric functions.

\subsection{Basis of $V^{(k)}$} \label{subsec:basis}

We will present a method to generate a basis of the space $V^{(k)}$, which will consist of single bases for the subspaces 
$V^{(k)}_{\Omega^{(\ell)}}$, $V^{(k)}_{\Gamma^{(j)}}$ and $V^{(k)}_{\Xi^{(r)}}$, defined in \eqref{eq:space_Omega_single}, \eqref{eq:space_Gamma_single} and 
\eqref{eq:subspace_singlevertex}, respectively. Thereby, the basis functions of 
$V^{(k)}_{\Gamma^{(j)}}$ and $V^{(k)}_{\Xi^{(r)}}$ will be constructed by means of \cite[Algorithm~1]{KaplVitrih2016}.

\paragraph{Basis of $V^{(k)}_{\Omega^{(\ell)}}$} We select for the basis the collection of $C^{2}$-smooth isogeometric functions
\begin{equation*} 
\bb x \mapsto
\begin{cases}
   (N_{\ab{i}}\circ (\ab{G}^{(\ell)})^{-1})(\bb x) \;
\mbox{ if }\f \, \bb x \in\Omega^{(\ell)},
\\ 0 \; \mbox{ otherwise},
\end{cases}
\end{equation*}
for all $(\ell,\ab{i}) \in \mathcal{I}^{(\ell)} \setminus (\mathcal{I}_{\Gamma} \cup \mathcal{I}_{\Xi})$, where $N_{\ab{i}}$ are the tensor-product 
B-splines of the space $\mathcal{S}_{k}^{d}$ (see Section~\ref{sec:settings}).

\paragraph{Basis of $V^{(k)}_{\Gamma^{(j)}}$} Clearly, a possible basis is determined by a minimal determining 
set of the spline coefficients $\{b_{\ab{i}}^{(\ell)} \}_{(\ell,\ab{i}) \in \mathcal{I}_{\Gamma^{(j)}}}$ with respect to the homogeneous linear 
system~\eqref{eq:nullspace} extended by the constraints $b_{\ab{i}}^{(\ell)}=0$ for $(\ell,\ab{i}) \in \mathcal{I} \setminus \mathcal{I}_{\Gamma^{(j)}}$. 
We use \cite[Algorithm 1]{KaplVitrih2016} to compute such a specific minimal determining set and denote it by $M^{(j)}$ (compare Section~\ref{subsec:vertex_space}). 

\paragraph{Basis of $V^{(k)}_{\Xi^{(r)}}$} Let us first define the following index space 
\[
\widetilde{\mathcal{I}}_{\Xi^{(r)}} = \mathcal{I}_{\Xi^{(r)}} \cup \bigcup_{j: v^{(r)} \in \Gamma^{(j)}} (\mathcal{I}_{\Gamma^{(j)}}\setminus \mathcal{I}_{M^{(j)}} ) .
\]
Since a possible basis of $V^{(k)}_{\Xi^{(r)}}$ is determined by a minimal determining set of the spline coefficients 
$\{b_{\ab{i}}^{(\ell)} \}_{(\ell,\ab{i}) \in \widetilde{\mathcal{I}}_{\Xi^{(r)}}}$ with respect to the homogeneous linear 
system~\eqref{eq:nullspace} extended by the constraints $b_{\ab{i}}^{(\ell)}=0$ for $(\ell,\ab{i}) \in \mathcal{I} \setminus \widetilde{\mathcal{I}}_{\Xi^{(r)}}$, 
we use again \cite[Algorithm 1]{KaplVitrih2016} to compute such a specific minimal determining set.

\paragraph{Nested spaces} We construct nested spaces $V^{(k)}$ by choosing $k = 2^{L} - 1$, $L \geq 0$, and denote 
the resulting isogeometric spaces by $V_h$, where $L$ is the level of refinement and $h = \mathcal{O}(2^{-L} )$. In addition, we consider the spaces 
$V_{2,0h} \subset V_{h}$, which are defined as
\[
 V_{2,0h} = \{w \in V_{h}| \; w(\ab{x}) = \frac{\partial w}{\partial \ab{n}}(\ab{x}) = \triangle w(\ab{x})  = 0  \mbox{ on }\partial \Omega\},
\]
where the conditions
\[
w(\ab{x}) = \frac{\partial w}{\partial \ab{n}}(\ab{x}) = \triangle w(\ab{x})  = 0  \mbox{ on }\partial \Omega
\]
for the isogeometric functions $w$ are the so-called homogeneous boundary conditions of order $2$.
The spaces $V_{2,0h}$ are determined by the homogeneous linear system, which is obtained by the homogeneous linear system~\eqref{eq:nullspace} and 
the additional constraints that all spline coefficients $b^{(\ell)}_{\ab{i}}$, corresponding to control points of the first three columns or rows of the 
boundary $\partial \Omega$, have to be zero. A basis of $V_{2,0h}$ can be constructed similarly as described above by using the extended linear system.  Since each function $w \in V_{h}$ ($w \in V_{2,0h}$) is piecewise $C^{\infty}$-smooth, we obtain $w \in H^{3}(\Omega)$ ($w \in H^{3}_0(\Omega)$), respectively.
This allows us to use the $C^{2}$-smooth geometrically continuous isogeometric functions to solve the triharmonic equation (see Example~\ref{ex:triharmonic}). 

\begin{rem}
Following the classification of the basis functions of the space $V^{(k)}$ into the two different kinds of basis functions presented in 
\cite[Section 4.1]{KaplVitrih2016} for the two-patch case, \emph{the basis functions of the first kind} are given by the basis functions of the space 
$V^{(k)}_{\Omega}$ and the \emph{basis functions of the second kind} are given by the basis functions of the spaces $V^{(k)}_{\Gamma}$ and $V^{(k)}_{\Xi}$.
This classification is based on the fact that the spline coefficients $b_{\ab{i}}^{(\ell)}$ of the basis functions of the space $V^{(k)}_{\Omega}$ are independent of the given 
multi-patch domain~$\Omega$, while the spline coefficients $b_{\ab{i}}^{(\ell)}$ of the basis functions of the spaces $V^{(k)}_{\Gamma}$ and $V^{(k)}_{\Xi}$ depend on 
the given multi-patch domain~$\Omega$. 
\end{rem}

\subsection{Examples}

We use the concept of isogeometric analysis to perform $L^{2}$ approximation and to numerically solve the triharmonic equation on two different bilinear multi-patch domains with 
extraordinary vertices. For both applications we consider the same model problems as introduced in \cite[Section 4.2 and 4.3]{KaplVitrih2016} for the case of two-patch 
domains and which can be simply extended to multi-patch domains. For more details about the isogeometric formulation of the two model problems, which is done via the weak 
formulation and Galerkin discretisation, we refer to \cite[Section 4.2]{KaplVitrih2016} for the case of $L^2$ approximation and to \cite[Section 4.3]{KaplVitrih2016} for 
the case of solving the triharmonic equation. 

In the first example we numerically verify the optimal approximation order of the space of $C^{2}$-smooth geometrically continuous isogeometric functions by means of 
$L^{2}$ approximation.
\begin{ex} \label{ex:fitting}
We consider the two bilinear multi-patch domains with extraordinary vertices given in Fig.\ref{fig:domains} (first column). One domain consists of 
three patches with an extraordinary vertex of valency three and the other one consists of five patches with an extraordinary vertex of valency five. We construct 
nested space $V_{h}$ for $L=0,1,\ldots, 4$. The corresponding bases are generated as described in Section~\ref{subsec:basis} for $L>1$ when $d=5$ and $L>0$ when $d=6$. 
For the other cases the basis construction slightly differs. We generate a minimal determining set for the merged space of the edge and vertex space, since the splitting into the two spaces (as presented in Section~\ref{sec:decompsition}) does not work for these cases. The number of resulting $C^{2}$-smooth isogeometric functions and 
the splitting into the number of functions of the patch space, edge space and vertex space are presented in Table~\ref{tab:fitting}. On both domains we approximate the function
\begin{equation} \label{eq:approx_func}
z(x_1,x_2) = 2 \cos (2 x_{1}) \sin (2 x_{2}), 
\end{equation}
see Fig.~\ref{fig:domains} (second column). Fig.~\ref{fig:convergenceL2appr} (top row) presents the resulting relative $H^0$-errors (i.e., $L^2$-errors) and the estimated convergence rates. The results indicate an optimal approximation order of $\mathcal{O}(h^{d+1})$. The estimated growth of the condition numbers of the diagonally scaled mass matrices (cf. \cite{Br95}), shown in  Fig~\ref{fig:convergenceL2appr} (bottom row), numerically confirms that our constructed bases are well conditioned.

\begin{table}
\centering\scriptsize
  \begin{tabular}{|c||c|c|c|c||c|c|c|c|} \hline
   & \# fcts & \# p.-fcts & \# e.-fcts & \# v.-fcts  & \# fcts & \# p.-fcts & \# e.-fcts & \# v.-fcts \\ \hline 
  & \multicolumn{4}{|c||}{$d=5$} & \multicolumn{4}{|c|}{$d=6$} \\ \hline \hline 
  $L$ & \multicolumn{8}{|c|}{Three-patch domain} \\ \hline
  0 & 52  & 27 & \multicolumn{2}{|c||}{25} & 82 & 48 & \multicolumn{2}{|c|}{34} \\ \hline
  1 & 145 & 108 & \multicolumn{2}{|c||}{37} & 247 & 192 & 9 & 46 \\ \hline
  2 & 493 & 432 & 15 & 46 & 865 & 768 & 51 & 46 \\ \hline
  3 & 1837 & 1728 & 63 & 46 & 3253 & 3072 & 135 & 46 \\ \hline
  4 & 7117 & 6912 & 159 & 46 & 12637 & 12288 & 303 & 46 \\ \hline \hline
$L$ & \multicolumn{8}{|c|}{Five-patch domain} \\ \hline
  0 & 81 & 45 & \multicolumn{2}{|c||}{36} & 131 & 80 & \multicolumn{2}{|c|}{51} \\ \hline
  1 & 236 & 180 & \multicolumn{2}{|c||}{56}  & 406 & 320 & 15 & 71 \\ \hline
  2 & 816 & 720 & 25 & 71 & 1436 & 1280 & 85 & 71 \\ \hline
  3 & 3056 & 2880 & 105 & 71 & 5416 & 5120 & 225 & 71 \\ \hline
  4 & 11856 & 11520 & 265 & 71 & 21056 & 20480 & 505 & 71 \\ \hline
  \end{tabular}
   \caption{The number of resulting $C^{2}$-smooth geometrically continuous isogeometric functions (\# fcts) divided into the number of functions of the patch space (\# p.-fcts), edge space (\# e.-fcts) and vertex space (\# v.-fcts) for the spaces $V_{h}$ with $L=0,1,\ldots,4$. 
   Note that for $L=0,1$ when $d=5$ and for $L=0$ when $d=6$ we cannot split the edge and vertex space as described in Section~\ref{sec:decompsition}.}
  \label{tab:fitting}
\end{table}

\begin{figure}[tbp]
\centering\footnotesize
\begin{tabular}{cccc}
Computational domain & $L^2$ approximation & Triharmonic equation  \\[0.2cm]
\multicolumn{3}{l}{Three-patch domain} \\
\includegraphics[width=4.5cm,clip]{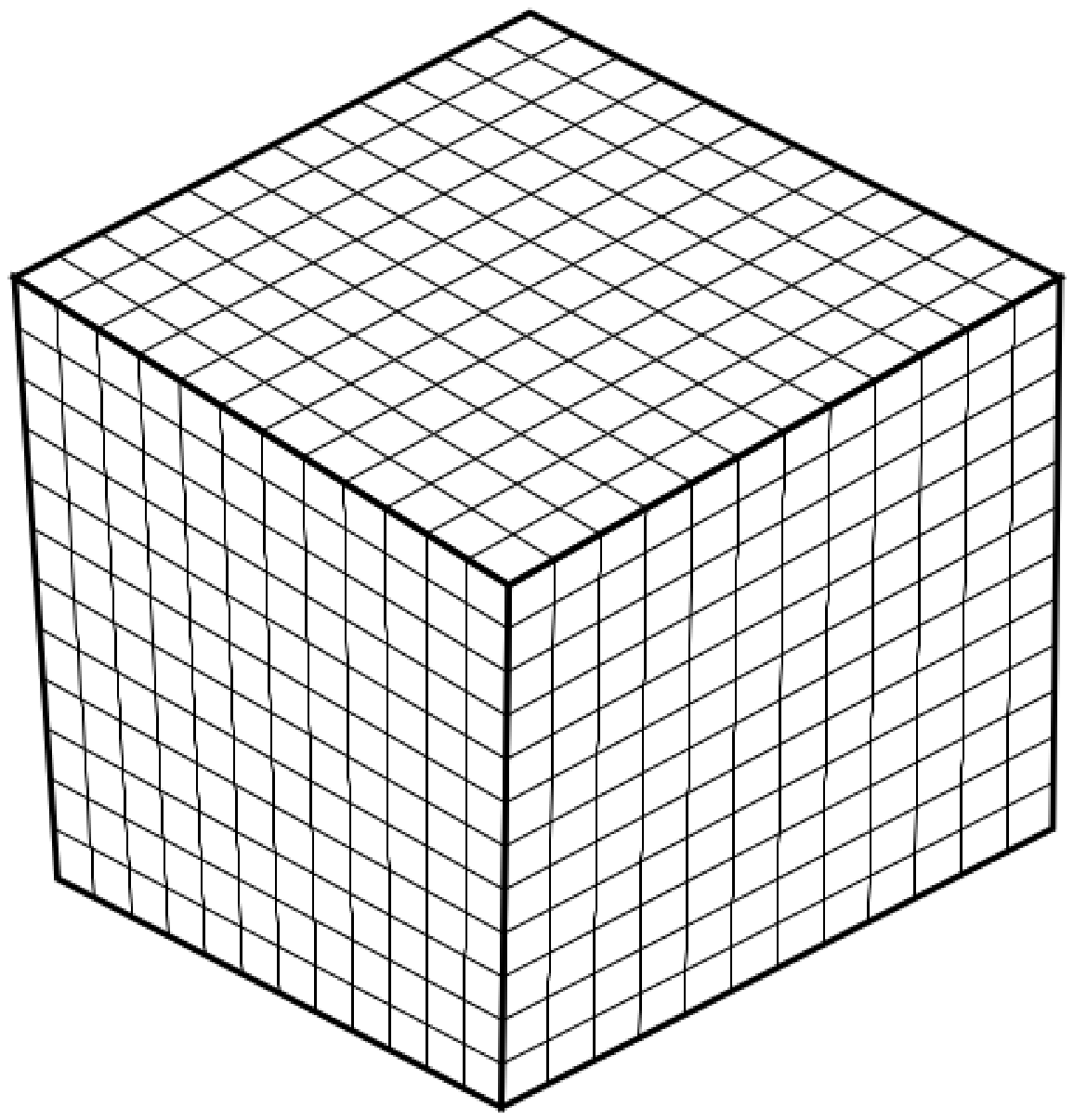}  &
\includegraphics[width=4.5cm,clip]{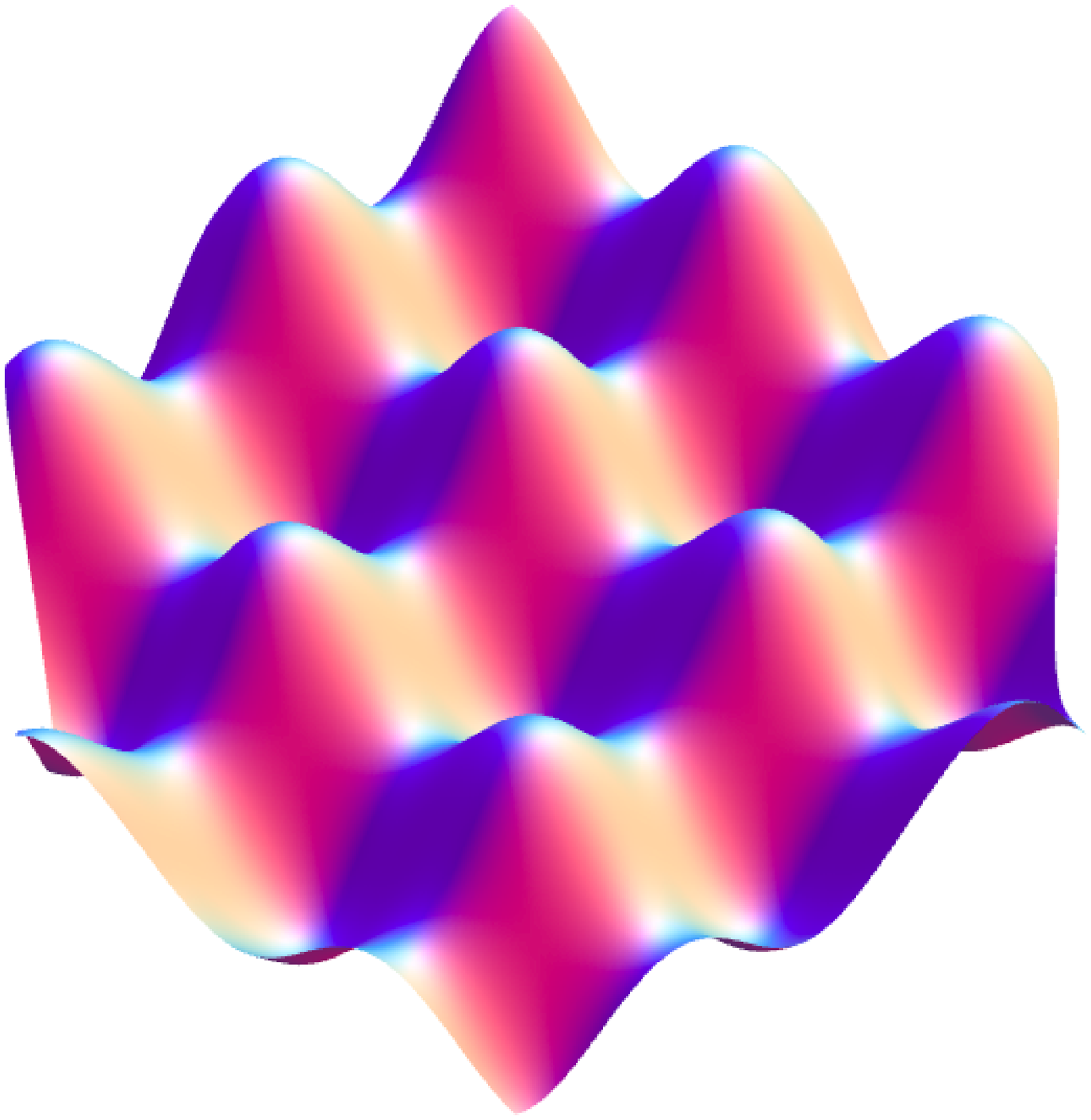} &
\includegraphics[width=4.5cm,clip]{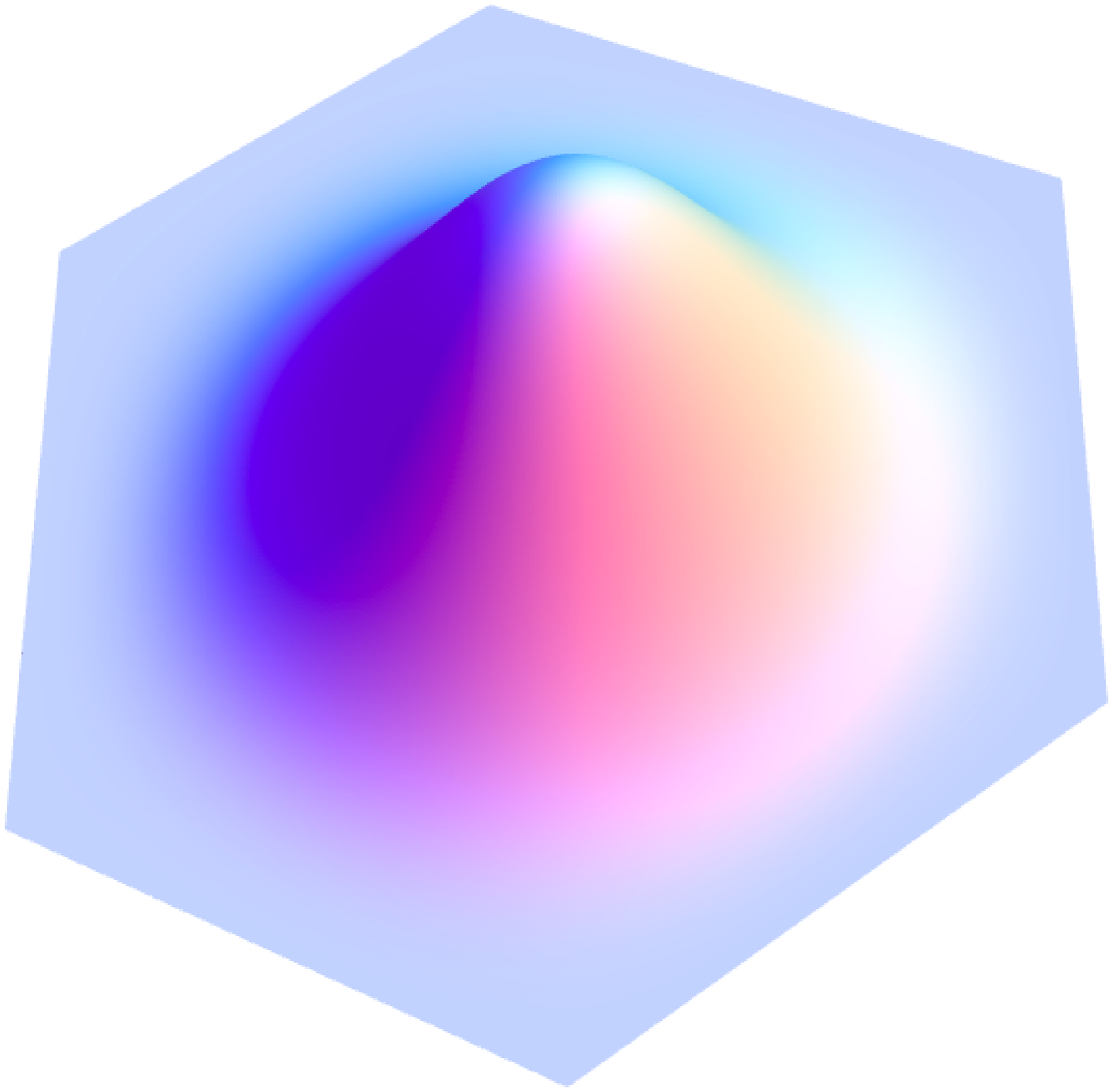} \\
\multicolumn{3}{l}{Five-patch domain} \\
\includegraphics[width=4.5cm,clip]{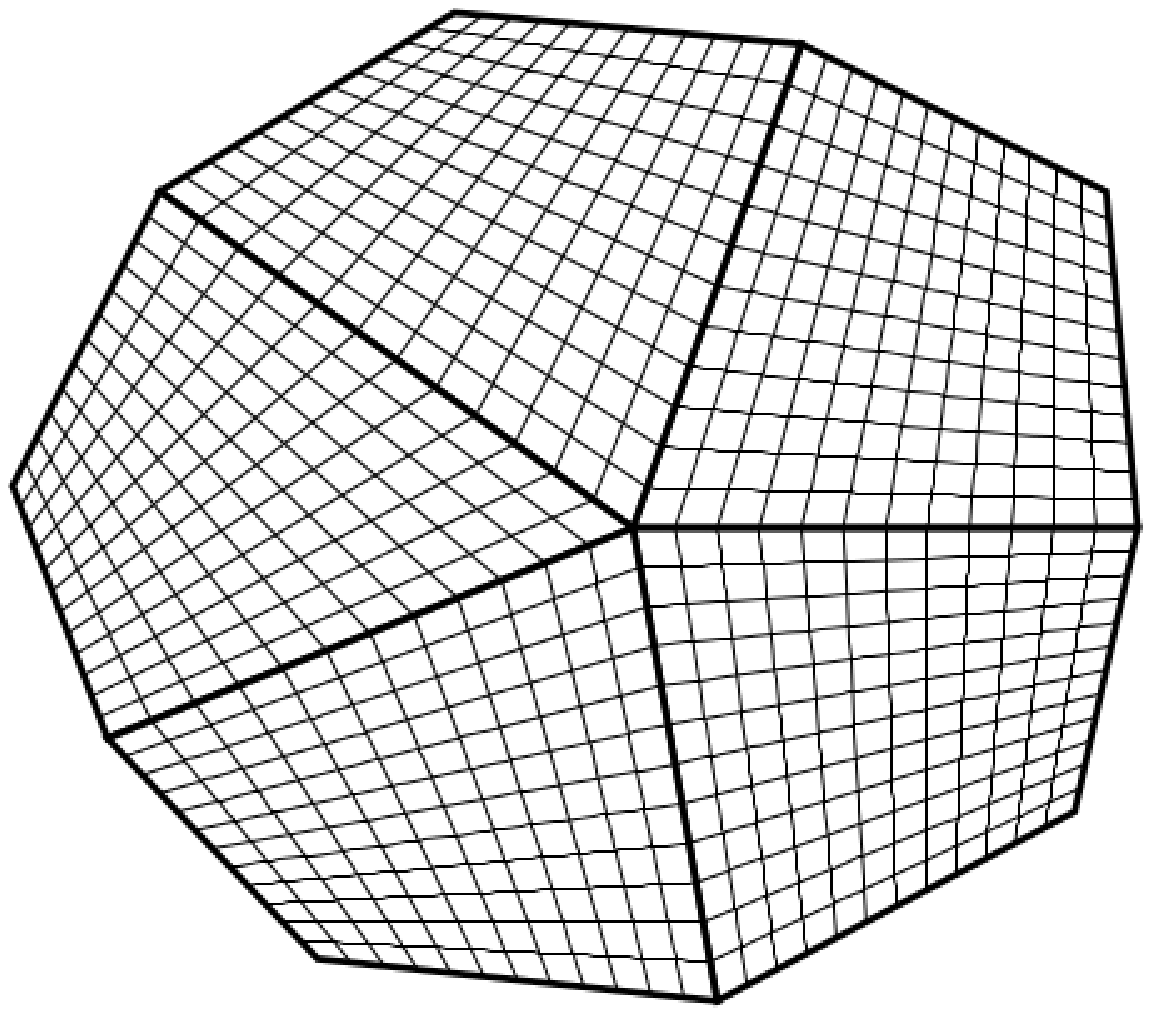}  &
\includegraphics[width=4.5cm,clip]{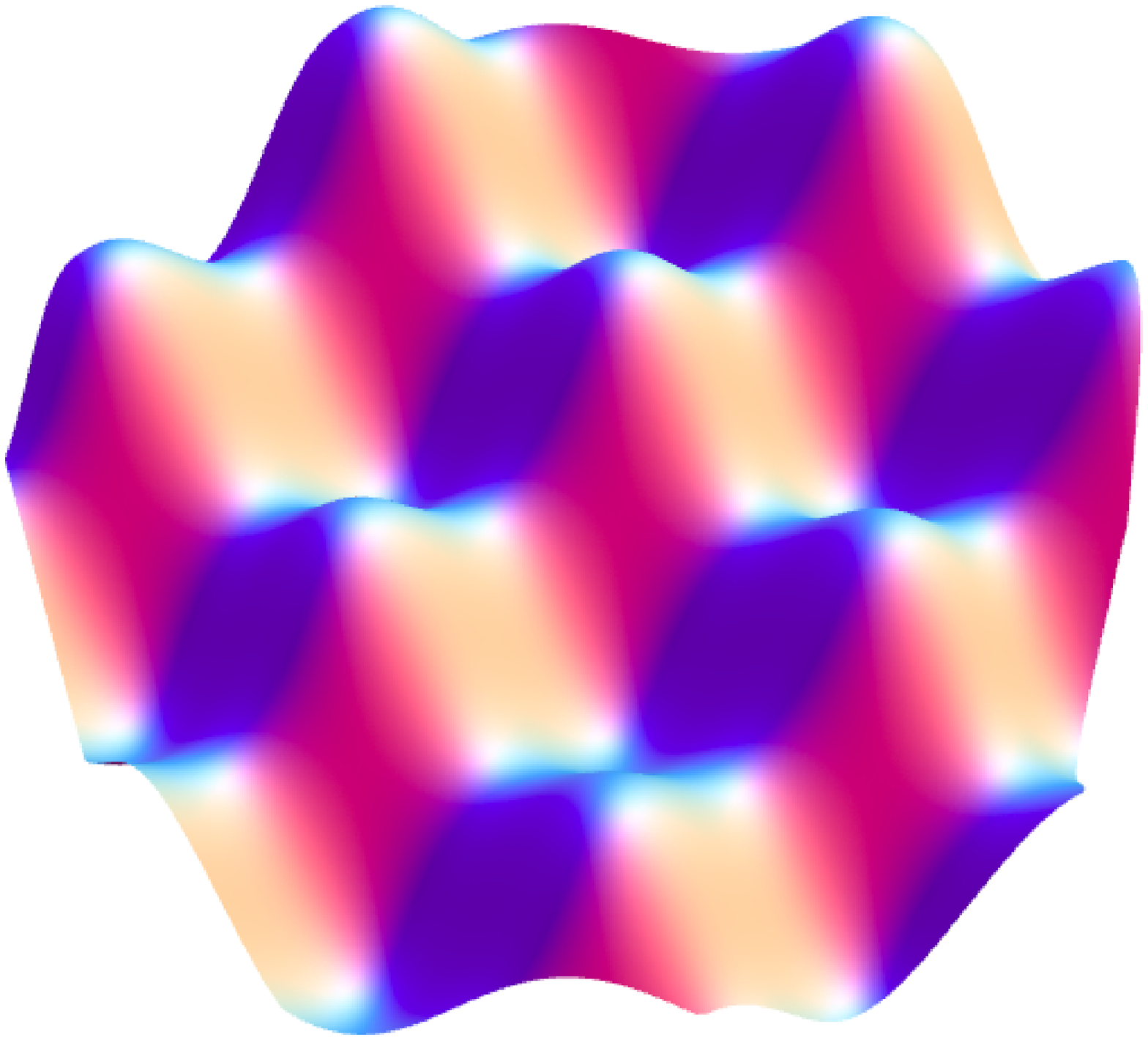} &
\includegraphics[width=4.4cm,clip]{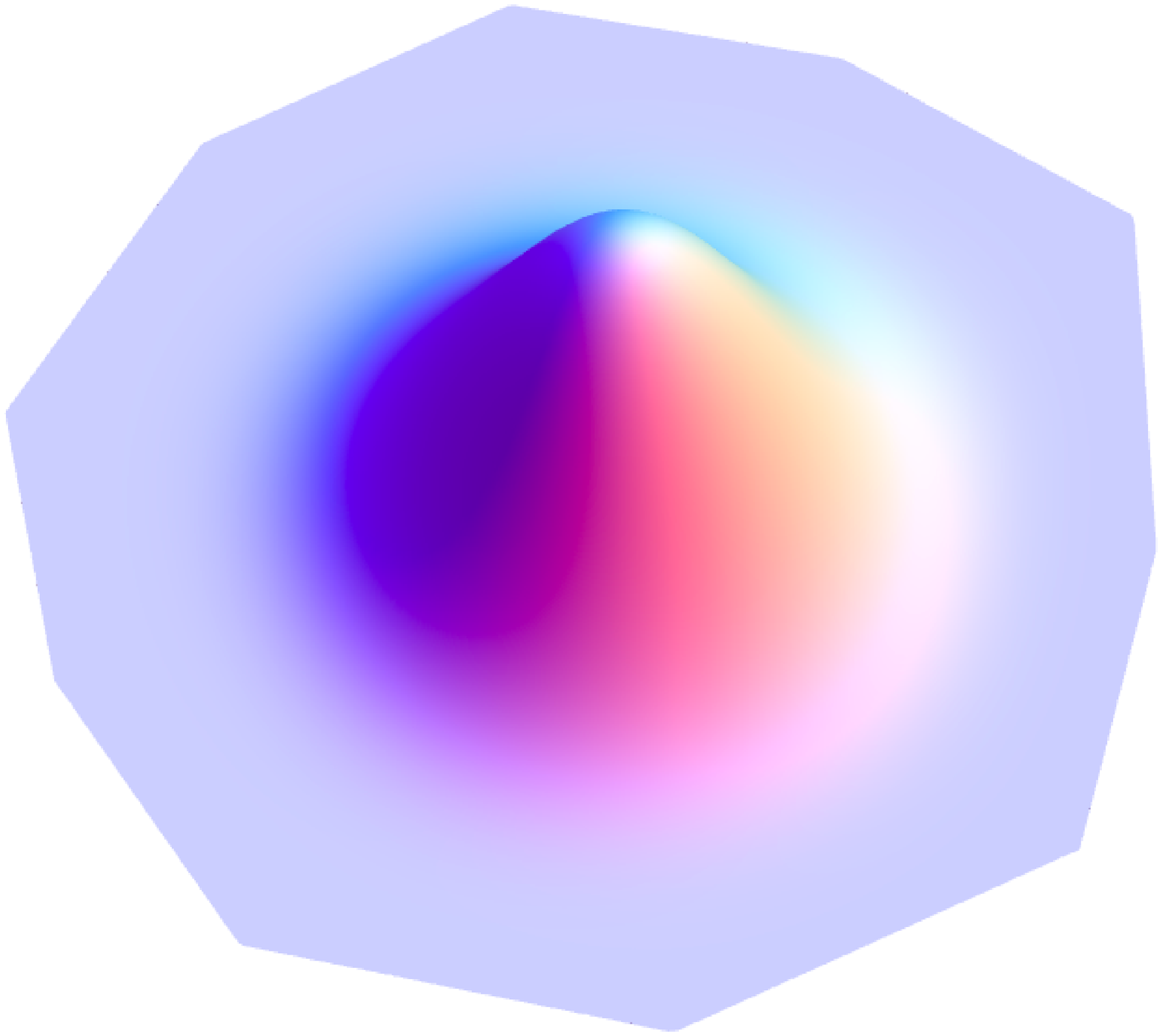}
\end{tabular}
\caption{The two bilinear multi-patch domains (first column) on which the function $z$ (second column), given in \eqref{eq:approx_func}, is approximated, and on 
which the triharmonic equation with the right side functions~$f$ obtained from the exact solutions $u$ (third column), given in \eqref{eq:rightside_three} and 
\eqref{eq:rightside_five}, is solved. See Example~\ref{ex:fitting} and Example~\ref{ex:triharmonic}.}
\label{fig:domains}
\end{figure}

 \begin{figure}[htb]
  \centering\footnotesize
  \begin{tabular}{cc}
  $d=5$ & $d=6$ \\[0.1cm]
  \multicolumn{2}{c}{\includegraphics[width=7.9cm,clip]{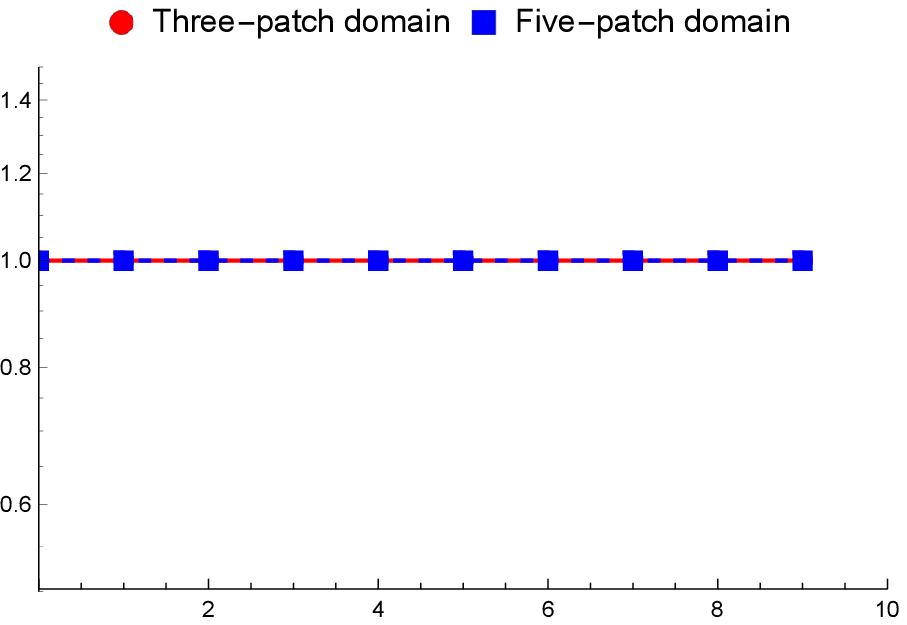}}\\[0.1cm]
   \includegraphics[width=7.2cm,clip]{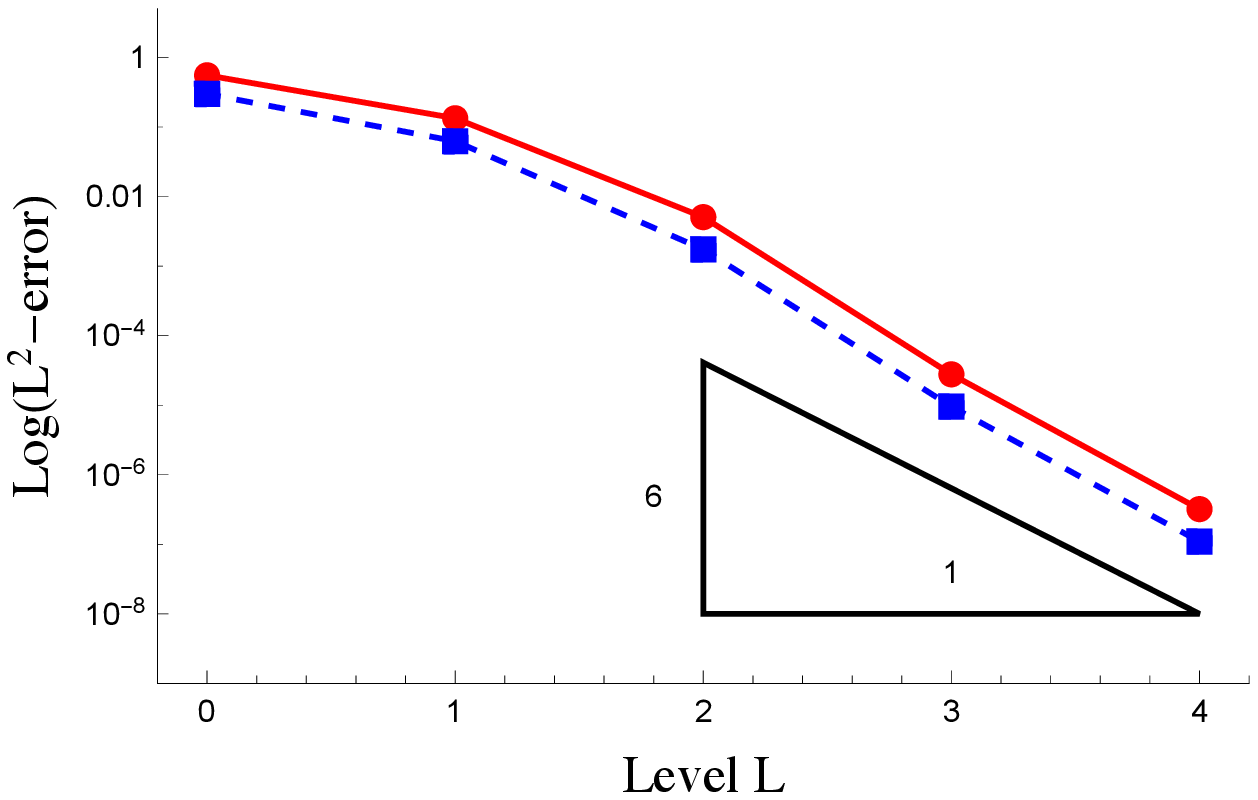}  & 
   \includegraphics[width=7.2cm,clip]{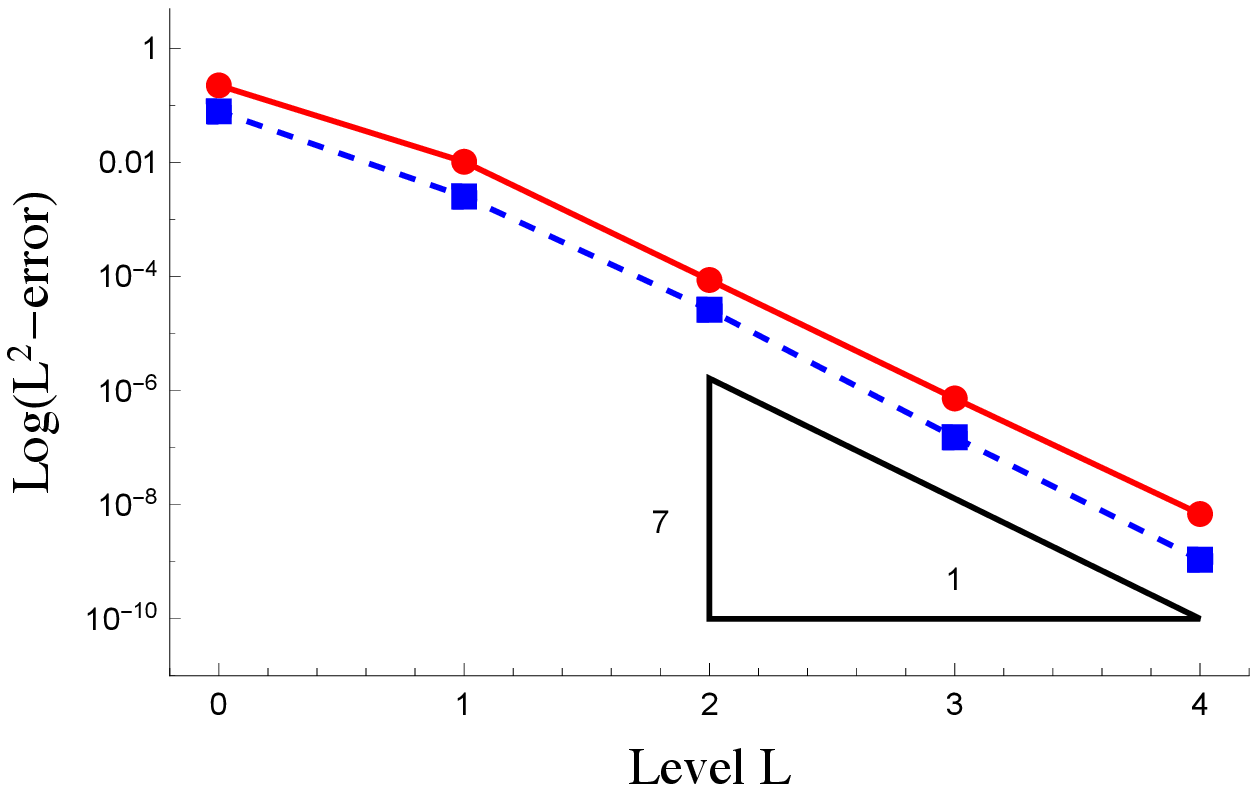}  \\[0.1cm]
   \includegraphics[width=7.2cm,clip]{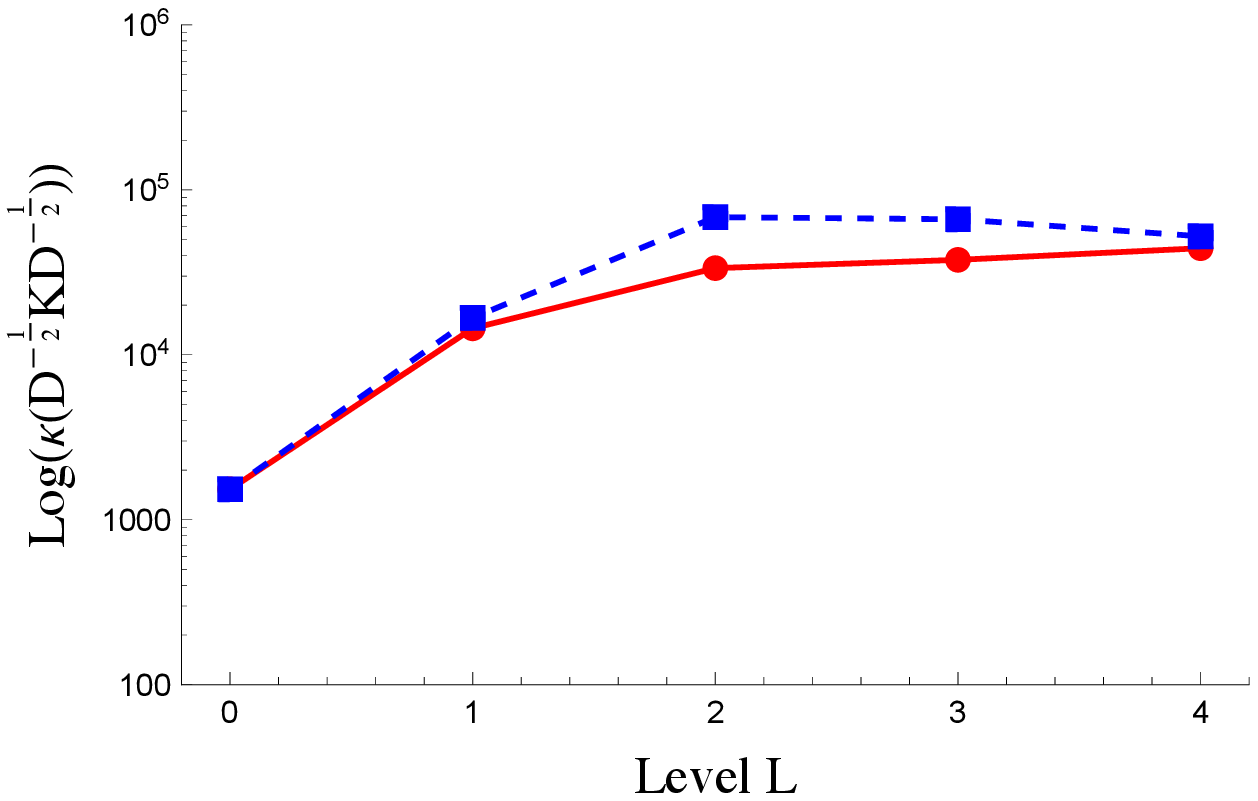}  & 
   \includegraphics[width=7.2cm,clip]{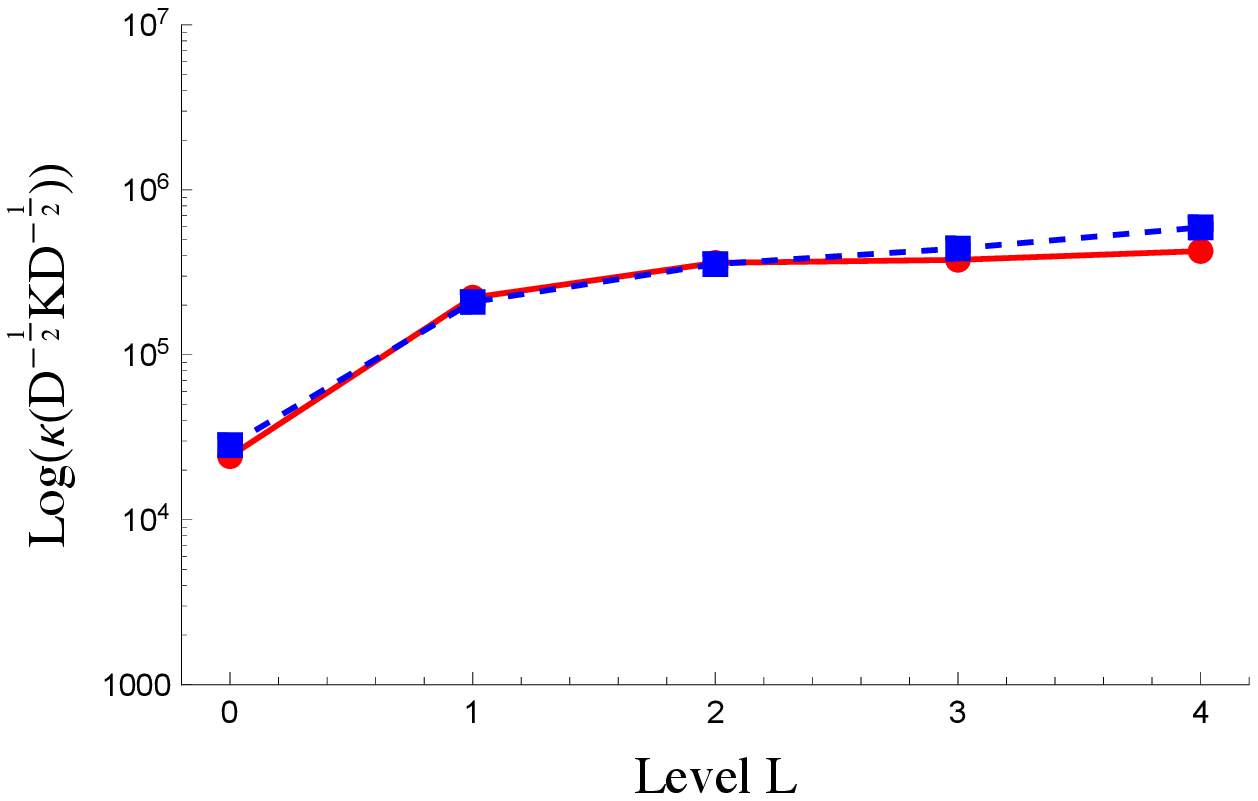}
  \end{tabular}
  \caption{The resulting relative $H^{0}$-errors with the corresponding estimated convergence rates (top row) and the estimated growth of the condition numbers (bottom row) by performing $L^2$-approximation on the two multi-patch domains shown in 
  Fig.~\ref{fig:domains} (first column). 
}
  \label{fig:convergenceL2appr}
  \end{figure}

\end{ex}

In the second example we demonstrate the potential of the globally $C^{2}$-smooth functions by solving a first example of a $6$-th order partial differential equation, 
given by the triharmonic equation, over multi-patch domains with extraordinary vertices.
\begin{ex} \label{ex:triharmonic}
We consider again the two bilinear multi-patch domains shown in Fig.~\ref{fig:domains} (first column) and construct for these domains the spaces 
$V_{2,0h}$ for $L=3$ as presented in Section~\ref{subsec:basis}. We use the constructed spaces to numerically solve the triharmonic equations with right side 
functions $f$ obtained from exact solutions fulfilling homogeneous boundary conditions of order~$2$, see Fig.~\ref{fig:domains} (third column). In case of the 
three-patch domain the considered exact solution is given by
\begin{align}
 \textstyle u(x_{1},x_{2}) = &\left(\frac{1}{2000000} \left(\frac{31}{8} - \frac{x_1}{2} - x_2 \right) \left(\frac{5507}{1440} + \frac{37 x_1}{72} - x_2\right) \left( \frac{477}{10} + 
    14 x_1 + x_2\right) \right. \nonumber \\[-0.2cm]
   &  \label{eq:rightside_three}  \\[-0.2cm]  
    & \left. \left(\frac{601}{155} + \frac{16 x_1}{31} + x_2 \right) \left(-\frac{1147}{308} + 
    \frac{39 x_1}{77} - 
    x_2 \right) \left(-\frac{1147}{4} + 77 x_1 - x_2\right) \right)^3 , \nonumber
\end{align}
and in case of the five-patch domain the considered exact solution is given by
\begin{align}
 \textstyle u(x_{1},x_{2}) = &  \left( \frac{1}{50000000} \left(\frac{512}{15} - 
\frac{32 x_1}{3} - x_2\right) \left(\frac{1558}{435} - \frac{14 x_1}{29} - x_2\right)
\left(\frac{174}{55} - \frac{x_1}{11} - x_2\right)  \right.\nonumber \\[0.1cm]
& \hspace{-1.2cm} \left(\frac{1661}{420} + 
\frac{17 x_1}{28} - x_2\right) \left(\frac{122}{15} + 2 x_1 -  x_2\right) \left(\frac{92}{9} + 
\frac{8 x_1}{3} + x_2\right)  \left(\frac{29}{6} + \frac{21 x_1}{20} + x_2 \right) \label{eq:rightside_five} \\[0.1cm]
& \hspace{-1.2cm}  \left. \left(\frac{839}{285} + \frac{2 x_1}{19} + x_2\right) 
\left(-\frac{279}{85} + \frac{9 x_1}{17} - x_2\right) 
\left(-\frac{72}{5} + \frac{9 x_1}{2} - x_2\right) \right)^3 . \nonumber
\end{align}
In Table~\ref{tab:triharmonic}  the number of resulting $C^{2}$-smooth functions and the splitting into the number of functions of the patch space, edge space and 
vertex space is presented. Furthermore, the table shows the resulting relative $H^{i}$-errors for $i = 0, 1,2,3$. 

\begin{table}
\centering\scriptsize
  \begin{tabular}{|c|c|c|c|c|c|c|c|c|} \hline
  $d$ & \# fcts & \# p.-fcts & \# e.-fcts & \# v.-fcts & $\s \frac{||u-u_{h}||_{0}}{||u||_{0}}$ & $\s \frac{||u-u_{h}||_{1}}{||u||_{1}}$ & $\s \frac{||u-u_{h}||_{2}}{||u||_{2}}$ & $\s \frac{||u-u_{h}||_{3}}{||u||_{3}}$ \\ \hline \hline
  \multicolumn{9}{|c|}{Three-patch domain} \\ \hline
  5 & 1399 & 1336 & 63  & 13 & 3.21e-5 & 4.3e-5 & 2.01e-4 & 2.39e-3\\ \hline
  6 & 2671 & 2523 & 135 & 13 & 9.01e-7 & 1.77e-6 & 1.45e-5 & 2.23e-4 \\ \hline \hline
  \multicolumn{9}{|c|}{Five-patch domain} \\ \hline
  5 & 2326 & 2205 & 121 & 16 & 4.33e-5 & 4.6e-5  & 2.43e-4 & 2.47e-3\\ \hline
  6 & 4446 & 4205 & 225 & 16 & 4.13e-7 & 9.34e-7 & 1.08e-5 & 1.71e-4\\ \hline 
  \end{tabular}
   \caption{Solving the triharmonic equation on the two multi-patch domains shown in Fig.~\ref{fig:domains} (first column) using the space $V_{2,0h}$ for $L=3$. 
   The number of resulting $C^{2}$-smooth geometrically continuous isogeometric functions (\# fcts) divided into the number of functions of the patch space (\# p.-fcts), 
   edge space (\# e.-fcts) and vertex space (\# v.-fcts), and the resulting relative $H^i$-errors, $i = 0,1,2,3$. %See Example~\ref{ex:triharmonic}.
    }
  \label{tab:triharmonic}
\end{table}

\end{ex}

\section{Conclusion}

We studied the space $V^{(k)}$ of biquintic or bisixtic $C^{2}$-smooth geometrically continuous isogeometric functions on a given bilinearly parameterized multi-patch domain~$\Omega \subset \R^2$. To investigate the dimension of $V^{(k)}$, we decomposed the 
space $V^{(k)}$ into the direct sum of three subspaces, which are called patch space, edge space and vertex space. This decomposition was selected since the computation of the dimension of the patch space is simple and the computation of the dimension of the edge space is based on the work for the two-patch case in \cite{KaplVitrih2016}. In order to find the 
dimension of the vertex space we had to distinguish between several cases depending on the 
type and valency of the single vertices.  

In addition, we presented first examples to use the space $V^{(k)}$ of globally 
$C^{2}$-smooth functions to solve a 6th-order partial differential equation, the so-called triharmonic equation, on multi-patch domains with extraordinary vertices by means of isogeometric discretisation. To obtain suitable test functions, we also described a method to generate a basis of the space $V^{(k)}$, which is based on the concept of finding a minimal determining set for the spline coefficients of the isogeometric functions. Moreover, we used the constructed basis functions to perform $L^{2}$ approximation on different bilinearly parameterized multi-patch domains. The obtained numerical results indicate optimal approximation order of the space $V^{(k)}$ and confirm that the generated bases are well conditioned.

The paper leaves several open issues which are worth to study. A first possible topic for future research is the theoretical investigation of the approximation power 
of the space of $C^{2}$-smooth geometrically continuous isogeometric functions. Another one is the implementation of further applications, e.g. the Phase-field crystal 
equation (cf. \cite{Gomez2012}), which require functions of such high smoothness. For both tasks local basis functions with small compact supports and explicit 
representations could be advantageous, which is in general not satisfied in our case. Our basis functions of the edge and vertex space possess a support over the whole 
part of one or even more common interfaces and are implicitly defined by means of a minimal determining set. Therefore, the finding of well-conditioned basis functions 
with the desired properties could be an important task for possible future research. 
Further interesting topics are e.g. the extension of our approach to non-bilinear multi-patch domains and to volumetric domains.  

\paragraph*{\bf Acknowledgement}
M.~Kapl was partially supported by the European Research Council through the FP7 ERC Consolidator Grant n.616563 HIGEOM, and by the 
Italian MIUR through the PRIN “Metodologie innovative nella modellistica differenziale numerica”. 
V.~Vitrih was partially supported by the Slovenian Research Agency (research program P1-0285).
These supports are gratefully acknowledged.

\appendix

\section{Proof of Lemma~\ref{thm:type2}}

We prove Lemma~\ref{thm:type2} for an inner vertex~$\ab{v}^{(0)}$ of valency $\nu \in \{4,5\}$. Let us first consider the case $\nu =4$.
Then $\Omega^{(2,3)}$ is the other two-patch domain for which the vertices $\bfm{v}_{\Omega^{(2, 3)}}^{(0)}, \, \bfm{v}_{\Omega^{(2, 3)}}^{(2)}$ and 
$\bfm{v}_{\Omega^{(2, 3)}}^{(4)}$ are collinear (see Fig.~\ref{fig:examples_typesVal4}). Otherwise, $\bfm{v}^{(0)}$ becomes a type $4$ vertex.
Therefore, we have $\psi_{2,4} = 0$.
The set $\bfm{e}^{\Xi^{(0)}}$ is of the form 
\[
\bfm{e}^{\Xi^{(0)}} = \{\tilde{e}^{(1)}, \bar{e}^{(1)}, \hat{e}^{(1)}, e^{(2)}, \tilde{e}^{(3)}, \bar{e}^{(3)}, \hat{e}^{(3)}, e^{(4)}\}. 
\]
By considering the supports of all equations in $\bfm{e}^{\Xi^{(0)}}$, we can conclude that $\hat{e}^{(1)}$ and $\hat{e}^{(3)}$ are linearly 
independent and linearly independent with all the other equations in $\bfm{e}^{\Xi^{(0)}}$. Moreover, $\tilde{e}^{(1)}, \bar{e}^{(1)},  e^{(2)}, 
\tilde{e}^{(3)}, \bar{e}^{(3)}$ are clearly linearly independent among them. 
To prove formula \eqref{eq:type2generalnu} we have to show that $e^{(4)}$ is a linear combination of $\tilde{e}^{(1)}, \bar{e}^{(1)}$,  $e^{(2)},  \tilde{e}^{(3)}, \bar{e}^{(3)}$. 
Let first $p_2 \neq 0$. Since $\bfm{v}^{(0)}, \bfm{v}^{(2)}$ and $\bfm{v}^{(4)}$ are collinear also $p_4 \neq 0$. Then
  $$
   e^{(4)} =\frac{2 p_4}{p_2} \psi_{3,4}^2 \tilde{e}^{(1)} + \frac{p_4}{p_2} \frac{\psi_{1,3}\,\psi_{3,4}}{\psi_{1,2}^2} \bar{e}^{(1)} + \left(\frac{p_4}{p_2}\right)^3 \frac{\psi_{1,4}}{\psi_{2,3}} e^{(2)} + 2 \left(\frac{p_4}{p_2}\right)^2 \frac{\psi_{1,2}^3}{\psi_{3,2}} \tilde{e}^{(3)} + \frac{p_4}{p_2} \frac{\psi_{1,2}^2 \psi_{1,3}}{\psi_{2,3}^3} \bar{e}^{(3)}.  
  $$
  Let now $p_2=0$. Then also $p_4=0$ and $q_2 \neq 0, q_4 \neq 0, p_3 \neq 0, p_1\neq 0$. We then obtain 
  $$
  e^{(4)} =\frac{2\, p_3^2 q_4^3}{q_2} \,\tilde{e}^{(1)} + \frac{p_3 q_4^2}{p_1^2 q_2^3} \,\psi_{1,3}\, \bar{e}^{(1)} 
  - \frac{p_1 q_4^4}{p_3 q_2^4}\, e^{(2)} + \frac{2 \,p_1 q_4^3}{p_3 q_2} \, \tilde{e}^{(3)} -
   \frac{p_1^2 q_4}{p_3^3 q_2^2} \,\psi_{1,3} \,\bar{e}^{(3)}.
  $$
  
Let us consider now the case $\nu = 5$. Similar to the proof for $\nu>5$, we assume equalities~\eqref{eq:psiForTwoLines} for $\nu =5$.
We have to consider equations 
\[
\bfm{e}^{\Xi^{(0)}}=\{\tilde{e}^{(1)}, \bar{e}^{(1)}, \hat{e}^{(1)}, \tilde{e}^{(2)}, \bar{e}^{(2)}, \hat{e}^{(2)}, e^{(3)}, e^{(4)}, e^{(5)}\},
 \]
among which $\hat{e}^{(1)}$ and $\hat{e}^{(2)}$ are clearly linearly independent with all other equations.
Matrix $A_\nu^{(2)}$, which is used in the proof of Lemma~\ref{thm:type2} for $\nu > 5$, is now $10\times 7$ matrix with columns corresponding to 
$\bfm{e}^{\Xi^{(0)}} \backslash \{\hat{e}^{(1)},\hat{e}^{(2)}\}$. 
We have to prove that
\begin{equation}  \label{eq:dependencyType2Val5}
e^{(5)} = \alpha_1 \,\tilde{e}^{(1)}  + \alpha_2 \,\bar{e}^{(1)} + \alpha_3 \,\tilde{e}^{(2)}  + \alpha_4 \,\bar{e}^{(2)} +\alpha_5\, e^{(3)}  + \alpha_6 \, e^{(4)}. 
\end{equation}
Recall \eqref{eq:Djnu}. By Lemma~\ref{lemma:Dj} and relations \eqref{eq:psiForTwoLines} rows $7,8$ and $9,10$ are pairwise linearly dependent, 
since $\det D_3^5$ involves the term $\psi_{2,5}$ and $\det D_{4,5}$ involves the term $\psi_{1,3}$.
Every row in $A_\nu^{(2)}$ contains precisely two nonzero elements, thus we can proceed as following.
From row 10 (or $9$) and from row $8$ (or $7$) we get 
$$
\alpha_6 = \frac{-\psi_{5,1}^3}{\psi_{3,4}\, \psi_{4,5}\, \psi_{5,3}} \quad {\rm and} \quad 
\alpha_5 = \frac{-\psi_{5,1}^3 \, \psi_{4,5}^2}{\psi_{2,3} \, \psi_{3,4}^2\, \psi_{4,2}\, \psi_{5,3}} .
$$
Then rows $6$ and $5$ give 
$$
\alpha_4 = \frac{\psi_{5,1}^3 \, \psi_{4,5}^2}{\psi_{1,2}^2 \, \psi_{2,3}^2\, \psi_{3,5}} \quad {\rm and} \quad 
\alpha_3 = \frac{2 \, \psi_{5,1}^3 \, \psi_{4,5}^2 \, \psi_{3,4}}{\psi_{1,2}\, \psi_{2,3} \, \psi_{4,2}\, \psi_{5,3}},
$$ 
respectively. Now we can express $\alpha_1$ and $\alpha_2$ both in two ways. From rows $4$ and $1$ we get
$$
   \alpha_2 =  \frac{\psi_{5,1} \, \psi_{4,5}^2 \, \psi_{3,4}}{\psi_{1,2}^2 \, \psi_{4,2}\, \psi_{5,3}} \quad {\rm and} \quad \alpha_2= - \frac{ \psi_{1,4}\, \psi_{4,5}\, \psi_{5,1}}{\psi_{1,2}^3}.
$$
Similarly from rows $3$ and $2$ we obtain
$$
   \alpha_1 =  - 2\frac{\psi_{5,1}^2 \, \psi_{4,5}^2 \, \psi_{2,3}}{\psi_{1,2}^2 \, \psi_{3,5}} \quad {\rm and} \quad 
   \alpha_1 = - 2\frac{ \psi_{4,5}^2\, \psi_{5,1}}{\psi_{1,2}}.
$$
In order to prove \eqref{eq:dependencyType2Val5} both expressions for $\alpha_1$ and $\alpha_2$ have to coincide, respectively. This is fulfilled if and only if
\begin{equation} \label{eq:equalities_alpha1_alpha2}
   \frac{\psi_{5,1} \, \psi_{2,3}}{\psi_{1,2} \, \psi_{3,5}} = 1, \quad
   - \frac{\psi_{1,2}\, \psi_{3,4}\, \psi_{4,5}}{\psi_{1,4}\, \psi_{4,2}\, \psi_{5,3}} =1.
\end{equation}
Both equalities in \eqref{eq:equalities_alpha1_alpha2} follow by a direct computation using \eqref{eq:psiForTwoLines} for $\nu=5$.

\section{Proof of Lemma~\ref{lem:type4}}

Relations \eqref{eq:val4type4_1_2} directly follow by comparing the corresponding four columns in $A_4^{(4)}$. Since 
$\displaystyle \frac{\psi_{1,4}\, \psi_{2,3} }{\psi_{4,3}\, \psi_{1,2}} = 1$, we get
\begin{align*}
  & \alpha_3 = \frac{2 \psi_{4,1}^3}{\psi_{2,3}}, \quad \alpha_2 = \frac{\psi_{3,4}\, \psi_{4,1}^3}{\psi_{2,3}^2\, \psi_{2,1}^2}, \quad 
  \alpha_1 =  \frac{2\, \psi_{2,3}\, \psi_{4,3}\, \psi_{4,1}^2}{\psi_{2,1}^2} = \frac{2\, \psi_{4,3}^2\, \psi_{1,4}}{\psi_{1,2}}, \\
  & \beta_3 = \frac{\psi_{4,1}}{2\,\psi_{4,3}\, \psi_{3,2}^2}, \quad \beta_2 = \frac{\psi_{3,4}^2\, \psi_{4,1}}{\psi_{1,2}\,\psi_{3,2}^2}, \quad
  \beta_1 = - \frac{\psi_{3,4}^2}{2\, \psi_{2,3}\, \psi_{1,4} \psi_{1,2}^2} = \frac{\psi_{3,4}}{2\, \psi_{1,2}^3}.
\end{align*}

Recall \eqref{eq:shape_points} and \eqref{eq:boundaryVertexVal2}. 
%To compute \textcolor{red}{particular} coefficients $\gamma_i$, $i=1,2,\ldots,9$, in 
In order to prove \eqref{eq:val4type4_3}, we can rotate and scale the domain $\Omega_{\bfm{v}^{(0)}}$ in order to get the following simplification 
$$
q_2 =0, \;  q_4 = 0, \; p_4 = 1.
$$ 
Then clearly $p_2 \neq 0$, $q_1 \neq 0$ and $q_3 \neq 0$.
Now we have to consider two cases regarding $p_1$. Let first $p_1\neq 0$, which implies $p_3 \neq 0$. Then

\begin{align*}
 &\gamma_1 = \frac{2 q_3 \left(p_1 \left(2 \tilde{q}_4-3 q_3\right)+p_3 \tilde{q}_1\right)}{p_1 p_2}, \\
 &\gamma_2 = \frac{p_3 \left(p_3^2 \left(2 p_2 \tilde{q}_1+3 \tilde{q}_2\right)+p_3 \left(p_2 \left(\left(5-2
   \tilde{p}_1\right) q_3-2 p_1 \tilde{q}_4\right)-3 \tilde{p}_2 q_3\right)+2 p_1 p_2 \left(\tilde{p}_4-1\right) q_3\right)}{p_1^2 p_2^4 q_3}, \\
    & \gamma_3 = \frac{p_3^2}{p_1^2 p_2^3}, \quad
  \gamma_4 = \frac{4 p_1 q_3 \left(p_3^2
   \left(-\tilde{q}_2\right)+p_3 \left(\left(\tilde{p}_2-p_2\right) q_3+p_1 p_2 \tilde{q}_4\right)-p_1 p_2 \left(\tilde{p}_4-1\right) q_3\right)}{p_2^4 p_3^2}, \\
     & \gamma_5 = \frac{p_1 \left(5 \left(p_2-1\right) q_3-2 p_2 \tilde{q}_4+3 \tilde{q}_3\right)+p_3 \left(2 \tilde{q}_2-3 p_2 \tilde{q}_1\right)}{p_2^5 p_3 q_3}, \quad 
   \gamma_6 =  -\frac{p_1}{p_2^5 p_3}, \\
      & \gamma_7 = \frac{2 p_1^2 q_3 \left(p_1 \left(\left(2-5 p_2\right) q_3+2 p_2 \tilde{q}_4-2 \tilde{q}_3\right)+3 p_2 p_3 \tilde{q}_1\right)}{p_2^2 p_3^3}, \\
         & \gamma_8 = \frac{p_1^3 \left(p_3 \left(p_2 \tilde{q}_4+2 \tilde{q}_3\right)-\left(2 \tilde{p}_3+p_2 \left(\tilde{p}_4-3\right)\right) q_3\right)}{p_2^3 p_3^3 q_3}, \quad
   \gamma_9 = \frac{p_1^3}{p_2^2 p_3^3}.
\end{align*}
Let now $p_1 = 0$, which implies $p_3 = 0$. Then 
\begin{align*}
 & \gamma_1 = \frac{2 q_3 \left(q_1 \left(2 \tilde{q}_4-3 q_3\right)+q_3 \tilde{q}_1\right)}{p_2 q_1}, \quad
 \gamma_2 = \frac{q_3 \left(p_2 \left(2 \left(\tilde{p}_4-1\right) q_1+\left(5-2 \tilde{p}_1\right) q_3\right)-3 \tilde{p}_2
   q_3\right)}{p_2^4 q_1^2}, \\
  & \gamma_3 = \frac{q_3^2}{p_2^3 q_1^2}, \quad 
   \gamma_4 =  \frac{4 q_1 \left(\tilde{p}_2 q_3-p_2 \left(\left(\tilde{p}_4-1\right) q_1+q_3\right)\right)}{p_2^4}, \quad  
   \gamma_6 = -\frac{q_1}{p_2^5 q_3}, \\
  & \gamma_5 = \frac{q_1 \left(5 \left(p_2-1\right) q_3-2 p_2 \tilde{q}_4+3 \tilde{q}_3\right)+q_3 \left(2 \tilde{q}_2-3 p_2 \tilde{q}_1\right)}{p_2^5 q_3^2}, \quad 
   \gamma_9 = \frac{q_1^3}{p_2^2 q_3^3},\\
  & \gamma_7 = \frac{2 q_1^2 \left(q_1 \left(\left(2-5 p_2\right)
   q_3+2 p_2 \tilde{q}_4-2 \tilde{q}_3\right)+3 p_2 q_3 \tilde{q}_1\right)}{p_2^2 q_3^2},
   \gamma_8 = -\frac{\left(2 \tilde{p}_3+p_2 \left(\tilde{p}_4-3\right)\right) q_1^3}{p_2^3 q_3^3}.
\end{align*}

\end{document}